\documentclass[12pt]{article}
\usepackage[T1]{fontenc}
\usepackage{amsmath, amsthm, amssymb}
\usepackage[ansinew]{inputenc}
\usepackage{amsfonts}
\usepackage{amssymb}
\usepackage{xcolor}
\usepackage{enumitem}
\usepackage{amsthm}
\usepackage{dsfont}
\usepackage{tikz-cd}
\DeclareMathAlphabet{\mathbbmsl}{U}{bbm}{m}{sl}

\def\min{\mathop{\rm min}}

\def\min{\mathop{\rm min}}

\def\R{\mathbb{R}^n}

\def\x{{\bf{x}}}

\newtheorem{theorem}{Theorem}[section]

\newtheorem{corollary}[theorem]{Corollary}
\newtheorem{example}[theorem]{Example}

\newtheorem{remark}[theorem]{Remark}
\newtheorem{lemma}[theorem]{Lemma}
\newtheorem{definition}[theorem]{Definition}
\newtheorem{proposition}[theorem]{Proposition}
\usepackage[numbers,sort&compress]{natbib}
\bibpunct[, ]{[}{]}{,}{n}{,}{,}

\begin{document}
\title{\large \textbf{Extended Horizontal Tensor Complementarity Problems}}
\author{Sonali Sharma$^{a,1}$, V. Vetrivel$^{a,2}$\\
{\small$^{a}$Department of Mathematics, IIT Madras, Chennai, India}\\
{\small $^{1}$Email: ma24r005@smail.iitm.ac.in}\\
{\small $^{2}$Email: vetri@iitm.ac.in}\\}
\date{}
\maketitle

\begin{abstract}
In this paper, we study the nonemptiness, compactness,  uniqueness, 
and finiteness  of the solution set of a new type of nonlinear complementarity problem, namely the extended horizontal tensor complementarity problem (EHTCP).  We introduce several classes of structured tensors and discuss the interconnections among these tensors. Consequently, we study the properties of the solution set of the EHTCP with the help of degree theory. 
\end{abstract}

\section{Introduction}
For a given ordered set of $k+1~(k \geq 1)$ matrices $\widehat{{\bf A}}:=\{{A}_{0}, {A}_{1},...,{A}_{k}\} \subseteq {\mathbb R}^{n\times n}$, a vector $q \in \mathbb{R}^{n}$ and an ordered set of $k-1$ positive vectors $\widehat{\bf d} := \{d_{1},d_{2},...,d_{k-1}\} \subseteq \mathbb{R}^{n}$, \textit{the extended horizontal linear complementarity problem} (for short, EHLCP), is to find vectors $x_{0},x_{1},...,x_{k}$ in $\mathbb{R}^{n}$ such that
\begin{equation*}
{{A}_{0} x_{0}} =~  {q} + {\displaystyle{\sum_{j =1}^k} {A_{j}x_{j}}},
\end{equation*}
\begin{equation*}
 x_{0} \wedge x_{1} = 0 \text{ and } (d_{j} - x_{j}) \wedge x_{j+1} = 0, 1\leq j\leq k-1,
\end{equation*}
where $'\wedge'$ is a pointwise minimum map. This problem is denoted by EHLCP$(\widehat{\bf A},\widehat{\bf d},{q})$. The EHLCP and the properties of its solution set with respect to several structured matrices have been extensively studied (see, for instance \cite{MR2690912,MR1340716,yadav2023generalizations}). It is easy to see that if $k =1$, the EHLCP reduces to the horizontal linear complementarity problem (HLCP). Moreover, by taking $A_{0} = I$ (the identity matrix), the HLCP becomes the standard linear complementarity problem (LCP). Due to the widespread applications of LCP, several generalizations of the LCP have been studied in literature (see \cite{MR3896653,MR3396730,GOWDA199597,mezzadri2021,gowda1996extended}). 
For a brief overview of the theory, applications and various numerical methods for the LCP and HLCP, the reader is referred to \cite{eaves1981,mezzadri2019,guler1995,zhang1994}.

Among other generalizations of the LCP, the tensor complementarity problem (TCP) has gained considerable interest in literature. For a given real tensor $\mathcal{A}$ of order $m$ and dimension $n$, and a given vector ${q} \in \mathbb{R}^{n}$, the TCP$(\mathcal{A},{q})$ \cite{MR3341670} is to find a vector $x$ in $\mathbb{R}^{n}$ such that
\begin{equation*}\label{E2}
{w} = \mathcal{A}{x}^{m-1}+q \in \mathbb{R}^{n} ~\text{and}~{x} \wedge {w} =0,
\end{equation*}  
where $\mathcal{A}{x}^{m-1}$ is a vector in $\mathbb{R}^{n}$, defined as 
\begin{equation*}\label{E3}
(\mathcal{A}{x}^{m-1})_{i} = \displaystyle{\sum_{i_{2},...,i_{m} =1}^{n} a_{ii_{2}...i_{m}}x_{i_{2}}...x_{i_{m}}},  
\end{equation*}
and $a_{i_{1}i_{2}...i_{m}} \in \mathbb{R}$ are the entries of $\mathcal{A}$ with $i_{1},i_{2},...,i_{m} \in \{1,2,...,n\}.$ The set of all $m$th order $n$ dimensional real tensors is denoted by $\mathbb{T}(m,n)$. The TCP has numerous applications in optimization, game theory, and absolute value equations. For instance, Huang et al. \cite{MR3613565} formulated a class of  $n$-person noncooperative game as a TCP, and a semismooth Newton method for finding the Nash equilibrium of this game was provided in \cite{chen2019finding}. For a detailed discussion on the theory, solution methods and applications of the TCP, one can see \cite{MR3989294,MR4023437,MR3998357} and references therein. Motivated by these applications of TCP, several generalizations of TCP  have been studied in recent years. One of these generalizations of TCP is the horizontal tensor complementarity problem (HTCP), which  was recently introduced by the authors \cite{Yadav30082024}, where the solution set properties of the HTCP were studied and an iterative method was proposed in \cite{Sun2024}.  Another generalization of TCP, known as the extended vertical tensor complementarity problem (EVTCP) was introduced in \cite{MR4044418}, and the bounds of the solution set of the EVTCP \cite{li2024global,wang2025bound} and the  properties of the solution set of the EVTCP with the help of some special structured tensors have been discussed \cite{li2024,li2025extended,meng2023existence}. Moreover, an application of the VTCP was studied in game theory, where the authors \cite {jia2024} provided the connection between the generalized multilinear games and the VTCP. 

It has been shown that various structured tensors play an important role in the study of the properties of the solution set of the TCP. For example, $R_{0}$ tensor \cite{MR3513267} and $P$ tensor \cite{MR3513266} give the compactness of the solution set, strictly semipositive tensor \cite{MR3501398} implies the existence of  solution for all $q$ and the uniqueness of the solution  when the vector $q$ is non-negative, the notion of non-degenerate tensors \cite{MR4310678} has been studied with regards to the finiteness of the solution set, strong $P$ tensor \cite{MR3513266} gives the uniqueness of the solution, etc. In order to study the properties of the solution set of the EVTCP, HTCP, and polynomial complementarity problem (PCP), the concept of these structured tensors has been extended and the properties of the corresponding solution set have been analysed  with the help of these structured tensors (see \cite{MR4044418,li2024,li2025extended,meng2023existence,Yadav30082024,
shang2023structured,Sun2024}). 

Motivated by the above generalizations of the TCP, in this paper, we introduce the {\it extended horizontal tensor complementarity problem} (EHTCP). For a given ordered set of $k +1~(k \geq 1)$ tensors $\widehat{\mathcal{A}} := \{\mathcal{A}_{0}, \mathcal{A}_{1},...,\mathcal{A}_{k}\}$ in $\mathbb{T}(m,n)$, a vector $q \in \mathbb{R}^{n}$ and an ordered set of $k -1$ positive vectors  $\widehat{\bf d} := \{d_{1},d_{2},...,d_{k-1}\} \subseteq \mathbb{R}^{n}$, ${\rm EHTCP}(\widehat{\mathcal{A}}, \widehat{\bf{d}},{q})$, is to find vectors $x_{0}, x_{1},...,x_{k}$ in $\mathbb{R}^{n}$ such that
\begin{equation*}
\mathcal{A}_{0} {x_{0}^{m-1}} = {q} + {\sum_{j = 1}^{k} \mathcal{A}_{j}{x_{j}^{m-1}}}, 
\end{equation*}
\begin{equation*}
{x_{0}} \wedge {x_{1}} = 0~\text{and}~({d_{j}}-{x_{j}}) \wedge {x_{j+1}} = 0,~j \in \{1,2,...,k-1\}. 
\end{equation*}
It is easy to observe that when $k =1$, the EHTCP reduces to the HTCP,
 which was recently studied in \cite{Yadav30082024,Sun2024}.
  In addition, if ${\mathcal A}_{0} = {\mathcal I}$ (the identity tensor), then EHTCP reduces to TCP \cite{MR3791481}. When $m=2$,  the EHTCP becomes the extended horizontal linear complementarity problem (EHLCP) studied in \cite{MR1340716, yadav2023generalizations}. The EHLCP further reduces to the horizontal linear complementarity problem (HLCP) \cite{zhang1994} and the standard linear complementarity problem \cite{MR3396730}. The relationship among these problems is illustrated as follows.
\begin{equation*}
\begin{tikzcd}[arrows=Rightarrow, row sep=huge]
\text{\bf {EHTCP}} \arrow[r,"k =1"] \arrow[d,swap,"m=2"] &
\text{HTCP} \arrow[r,"\mathcal{A}_{0} = \mathcal{I}"] \arrow[d,swap,"m=2"]  &
\text{TCP} \arrow[d,"m=2"]
\\
\text{EHLCP} \arrow[r,"k =1"] & \text{HLCP}  \arrow[r,"{A}_{0} = {I}"] & \text{LCP}
\end{tikzcd}
\end{equation*} 
As EHTCP serves as a unifying framework for various complementarity problems, we are motivated to investigate the properties of its solution set. To this end, we introduce several classes of structured tensors and employ degree-theoretic tools to analyse the characteristics of the  solution set of the EHTCP. 
In this paper, we consider the following problems related to the EHTCP.
\begin{enumerate}
\item[\rm(i)] In HTCP, the $R_{0}$ tensor pair gives the compactness of the solution set  (see \cite[Theorem 3.3]{Yadav30082024}). 
This motivates us to generalize the concept of $R_{0}$ tensor pair to the EHTCP and study the compactness of the corresponding solution set. 
\item[\rm(ii)] The $P$ tensor pair and the strong $P$ tensor pair provide the existence, compactness and uniqueness of the solution of the HTCP, respectively (see \cite[Theorems 3.18 and 3.23]{Yadav30082024}). 
Motivated by this, one can ask, whether  these concepts can be extended to the EHTCP or not? If so, then can we expect the same kind of outcome in the EHTCP?
\item[\rm(iii)]  
The finiteness of the solution set of the TCP was investigated in \cite{MR4310678} for a class of non-degenerate tensors. The authors showed that a tensor being non-degenerate is not equivalent to the finiteness of the solution set of the corresponding TCP. They also defined non-degenerate property and provided the finiteness of the solution set of the TCP when the involved tensor has this property. However, this result has not yet been explored in the context of HTCP. This gap motivates us to extend the notion of non-degenerate tensors to the EHTCP, and to study the finiteness of the corresponding solution set. 
\end{enumerate}
In order to investigate the aforementioned problems, we extend  the concepts of the  ${R}_{0}$ tensor pair, $P$ tensor pair, and strong $P$ tensor pair from the HTCP to the EHTCP. We refer to these generalizations as $\mathbb{EHR}_{0}$ tensors, $\mathbb{EHP}$ tensors and strong $\mathbb{EHP}$ tensors, respectively. Utilizing these structured tensors and degree theory, we study the nonemptiness, uniqueness and compactness of the solution set of the EHTCP. To address the third problem, we  introduce the notions of  $\mathbb{EHND}$ tensors and strong $\mathbb{EHND}$ tensors, and subsequently investigate the finiteness of the EHTCP solution set in relation to these tensor classes. In particular,  we establish that the solution set of the EHTCP is finite when the involved set of tensors is a strong $\mathbb{EHND}$ tensor.  

The outline of  this paper is as follows: In Section \ref{Section 2}, we introduce some basic notation, definitions and results that will be useful in the sequel. The EHTCP is defined in Section \ref{Section 3}, and  several structured tensors related to the EHTCP are introduced and the interconnections among them are discussed in Subsection \ref{Subsection 3.1}. In Section \ref{Section 4}, we explore the properties of the solution set of the EHTCP.  The degree of EHTCP is defined in Subsection \ref{Degree of EHTCP}, existence results are obtained in Subsection \ref{Existence and Compactness}, and  finiteness and uniqueness results are discussed in Subsection \ref{Finiteness and Uniqueness}. In Section \ref{section 5} we draw a conclusion of our work. 

\section{Preliminaries}\label{Section 2}

Throughout this paper, we will use the following notation:
\begin{enumerate}
\item[\rm(i)] For a natural number $k$, the set $\{1,2,...,k\}$ is denoted by $[k]$. 
\item[\rm(ii)] The $n$-dimensional Euclidean space with the usual inner product is denoted by $\mathbb{R}^{n}.$ The collection of all the vectors ${x}$ in $\mathbb{R}^{n}$ such that $x_{i} \geq 0~ (> 0)$, for all $i \in [n]$ is denoted by $\mathbb{R}^{n}_{+} (\mathbb{R}^{n}_{++}).$ 
\item[\rm(iii)] ${x} \wedge {y}$ denotes the vector $\min\{x,y\}$ having its $i$th component as $\min\{{x_{i}, y_{i}}\},$ and ${x} \ast {y}$ denotes the Hadamard (= componentwise) product of $x$ and $y$. For any ${x} \in \mathbb{R}^{n}$ and a positive integer $m$,  ${x}^{[m]}$ denotes a vector in $\mathbb{R}^{n}$ with its $i$th component as $x^{m}_{i}$. 
\item[\rm(iv)]  The set $\mathbb{T}(m,n)$ denotes the set of all real tensors of order $m$ and dimension $n$, and the elements of $\mathbb{T}(m,n)$ are denoted by math calligraphic letters, such as $\mathcal{A}_{0}, \mathcal{A}_{1},..$ and ${\mathcal I}$ denotes  the identity tensor. The collection of all real $n \times n$ matrices is denoted by $\mathbb{R}^{n \times n}$ and $A,B,...$ are used to denote the elements in $\mathbb{R}^{n \times n}$. 
\item[\rm(v)] The $k$-ary Cartesian power of ${\mathbb R}^{n}$ is denoted by $\Theta^{(k)}_n.$ The bold zero `${\bf 0}$' denotes a vector in $\Theta^{(k)}_{n}$ such that ${\bf 0} = (0,0,...,0) \in \Theta^{(k)}_{n}.$ The $k$-ary Cartesian power of ${\mathbb R}^{n}_{++}$  is denoted by $\Theta^{(k)}_{n,++}.$ The set $\Theta^{(k)}_{(m,n)}$ denotes the $k$-ary Cartesian power of the set $\mathbb{T}(m,n)$.
\end{enumerate}

We recall the following property of the vector $\min\{x,y\}$, for any $x,y$ in $\mathbb{R}^{n}.$ 
\begin{proposition}\label{pp1}{\rm \cite[Proposition 1]{MR3896653}}
Let $x,y,z \in \mathbb{R}^{n}$. The following statements  are valid.
\begin{enumerate}
\item[\rm(a)] $z + (x \wedge y) = (z+x) \wedge (z+y),$
\item[\rm(b)] Consider the following statements: 
\begin{enumerate}
\item[\rm(i)] ${x} \wedge {y} = 0,$
\item[\rm(ii)] ${x} \geq 0, {y} \geq {0},~\text{and}~{x} \ast {y} = 0,$
\item[\rm(iii)] ${x} \geq 0, {y} \geq {0},~\text{and}~\langle {x}, {y} \rangle = 0.$
\end{enumerate}
Then, $(i) \iff (ii) \iff (iii).$
\end{enumerate}
\end{proposition}

\begin{definition}\rm \cite{MR3341670,MR3821072,MR4310678}
A tensor $\mathcal{A} \in \mathbb{T}(m,n)$ is said to be a/an
\begin{enumerate}
\item[\rm(i)] $R_{0}$ tensor if ${x} \wedge {\mathcal{A}x^{m-1}} = 0 \implies {x} = 0$.
\item[\rm(ii)] $P$ tensor if for any nonzero $x \in \mathbb{R}^{n}$, there exists $k \in [n]$ such that $$x_{k}(\mathcal{A}{x^{m-1}})_{k} > 0.$$
\item[\rm(iii)] strictly semipositive tensor if for any nonzero ${x} \in \mathbb{R}^{n}_{+}$ there exists $k \in [n]$ such that $$x_{k}> 0~\text{and}~(\mathcal{A}{x^{m-1}})_{k} > 0.$$ 
\item[\rm(iv)] non-degenerate tensor if ${x} \ast \mathcal{A}{x}^{m-1} = 0 \implies {x} = {0}$.
\end{enumerate} 
\end{definition}

\subsection{Degree Theory}
In this paper, we employ  degree theoretic tools to establish the existence results for the $\rm{EHTCP}$. All the necessary results concerning the degree theory are given in \cite{MR1955649} (particularly Theorem 2.1.2 and Proposition 2.1.3); see also, \cite{MR4576574,MR2222819}. Here is a short review. Let  $U$ be a nonempty, bounded, open subset of ${\R}$, and let ${\bar{U}}$, $\partial{U}$ denote the closure and boundary of $U$, respectively. Let $\Phi : \bar{U} \to {\R}$ be a continuous function and ${\bf u} \notin \Phi(\partial{U})$. Then the degree of $\Phi$ over $U$ with respect to ${\bf u}$ is defined. It is denoted by $\deg(\Phi,U,{\bf u})$ and it is always an integer. If $\deg(\Phi,U,{\bf u}) \neq 0$, then $\Phi({\bf w}) = {\bf u}$ has a solution in $U$. Suppose that $\Phi({\bf w}) = {\bf u}$ has a unique solution, say ${\bf w}^{\ast} \in U.$ Then $\deg(\Phi,U,{\bf u})$ is invariant for any bounded open set $U^{'}$ containing ${\bf w}^{\ast}$ and contained in $U$, and we denote $\deg(\Phi,U^{'},{\bf w})$ as $\deg(\Phi,{\bf w})$. The following properties hold for the $\deg(\Phi,U,{\bf u})$. 
\begin{enumerate}
\item[\rm(i)] $\deg({\rm I},U,{\bf u}) = 1$ if ${\bf u} \in U$, where $\rm I$ denotes the identity function.
\item[\rm(ii)] {\bf (Nearness Property)}. Let $\Psi : {\R} \to {\R}$ be a continuous function such that $\sup\{\|\Phi({\bf w})-\Psi({\bf w})\|_{\infty}: {\bf w} \in \bar{U}\}$ is sufficiently small. Then $\deg(\Phi,U,{\bf u}) = \deg(\Psi,U,{\bf u})$, where $\|\cdot\|_{\infty}$ denotes the max-norm of vectors in ${\R}$.
\item[\rm(iii)] {\bf (Homotopy Invariance Property)}. 
Let $\mathcal{Z}({\bf w},\theta) : {\R} \times [0,1] \to {\R}$ be a homotopy. The set $\Delta = \{{\bf w}\in {\R} : \mathcal{Z}({\bf w},\theta) = {0}~\text{for some}~0\leq\theta\leq1\}$ is said to be the set of zeros for $\mathcal{Z}$. Suppose that $\Delta$ is bounded. Let $S$ be a bounded set such that $\Delta \subseteq S$, then we have 
$$\deg(\mathcal{Z}({\cdot},0), S,{0})= \deg(\mathcal{Z}({\cdot},1), S,{ 0}).$$
\end{enumerate}
We note that all the degree theoretic results and concepts are also applicable over any finite dimensional Hilbert space. 

\section{The Extended Horizontal Tensor Complementarity Problem}\label{Section 3}

In this section, we formally introduce the extended horizontal tensor complementarity problem (EHTCP), and then define several structured tensor classes related to the EHTCP.
\begin{definition}\rm 
Let $\widehat{\mathcal{A}} = (\mathcal{A}_{0}, \mathcal{A}_{1},...,\mathcal{A}_{k}) \in \Theta^{(k+1)}_{(m,n)},$~ $\widehat{{\bf d}} = (d_{1},d_{2},...,d_{k-1}) \in \Theta^{(k-1)}_{n,++}$, and $q \in \mathbb{R}^{n}$, where $k \geq 1$. The \textit{extended horizontal tensor complementarity problem} (for short, EHTCP), denoted by ${\rm EHTCP}(\widehat{\mathcal{A}}, \widehat{\bf{d}},{q})$, is to find vectors $x_{0}, x_{1},...,x_{k}$ in $\mathbb{R}^{n}$ such that
\begin{equation}\label{EHTCP1}
\mathcal{A}_{0} {x_{0}^{m-1}} = {q} + {\sum_{j = 1}^{k} \mathcal{A}_{j}{x_{j}^{m-1}}}, 
\end{equation}
\begin{equation}\label{EHTCP2}
{x_{0}} \wedge {x_{1}} = 0~\text{and}~({d_{j}}-{x_{j}}) \wedge {x_{j+1}} = 0,~j \in [k-1]. 
\end{equation}
\end{definition}
The set of all the vectors ${\bf x} = (x_{0},x_{1},...,x_{k}) \in \Theta^{(k+1)}_{n}$ satisfying Eqs. (\ref{EHTCP1}) 
 and (\ref{EHTCP2}) is said to be the solution set of the  ${\rm EHTCP}(\widehat{\mathcal{A}}, \widehat{\bf{d}},{q})$, and is denoted by ${\rm SOL}(\widehat{\mathcal{A}}, \widehat{\bf{d}},{q})$.
The following observation is immediate for ${\rm SOL}(\widehat{\mathcal{A}}, \widehat{\bf{d}},{q})$ and will be useful in the sequel. For the sake of completeness, we provide a proof here (see also \cite[Lemma 3.1]{yadav2023generalizations}).   
\begin{lemma}\rm\label{solution}
Let $\widehat{\mathcal{A}} = (\mathcal{A}_{0}, \mathcal{A}_{1},...,\mathcal{A}_{k}) \in \Theta^{(k+1)}_{(m,n)}.$ If ${\bf x} = (x_{0},x_{1},...,x_{k}) \in {\rm SOL}(\widehat{\mathcal{A}}, \widehat{\bf{d}},{q})$, then ${\bf x}$ satisfies 
$$\mathcal{A}_{0}x_{0}^{m-1} = {q} + {\sum_{j =1}^{k} \mathcal{A}_{j}x_{j}^{m-1}}~\text{and}~x_{0} \wedge x_{i} = 0~\forall~i \in [k].$$
\end{lemma}
\begin{proof}
Let ${\bf x} = (x_{0},x_{1},...,x_{k}) \in {\rm SOL}(\widehat{\mathcal{A}}, \widehat{\bf{d}},{q}).$ This gives
\begin{equation*}\label{EH123}
\mathcal{A}_{0} {x_{0}^{m-1}} = {q} + {\sum_{j = 1}^{k} \mathcal{A}_{j}{x_{j}^{m-1}}}, 
\end{equation*}
\begin{equation}\label{EH124}
{x_{0}} \wedge {x_{1}} = 0~\text{and}~({d_{j}}-{x_{j}}) \wedge {x_{j+1}} = 0,~j \in [k-1]. 
\end{equation} 
From Eq. (\ref{EH124}) and Proposition \ref{pp1}, we get $x_{i} \geq {0}$ for all $i \in [k] \cup \{0\}.$ Now, it is enough to prove that $x_{0} \ast x_{2} = 0.$ Assume on the contrary that there exists an index $i \in [n]$ such that $({x_0})_{i}(x_{2})_{i} \neq 0.$ 
From Eq. (\ref{EH124}), we get $(x_{1})_{i} =0$ and $(d_{1}-x_{1})_{i} (x_{2})_{i} = 0.$ As $(x_{1})_{i} =0$ and  $({x_2})_{i} \neq 0$, we get $(d_{1})_{i} = 0$ leading to a contradiction as $d_{1} \in \mathbb{R}^{n}_{++}.$ Hence $x_{0} \ast x_{2} = 0.$ Similarly  $x_{0} \ast x_{i} = 0$ for all $i \in [k]$ and our conclusion follows from Proposition \ref{pp1}.  
\end{proof}

\subsection{Special Structured Tensors Related to the EHTCP}\label{Subsection 3.1}
In the following, we introduce some special structured tensors related to the EHTCP and  analyse the interconnections among these tensors.

\begin{definition}\rm\label{special tensors}
Let $\widehat{\mathcal{A}} = (\mathcal{A}_{0}, \mathcal{A}_{1},...,\mathcal{A}_{k}) \in \Theta^{(k+1)}_{(m,n)}.$ We say that $\widehat{\mathcal{A}}$ is a/an  
\begin{enumerate}
\item \label{R_0 tensor}  extended horizontal ${R}_{0}$ ($\mathbb{EHR}_{0}$) tensor if 
\begin{equation*}
\left.
\begin{aligned}
\mathcal{A}_{0}x_{0}^{m-1} = {\sum_{j =1}^{k} \mathcal{A}_{j}x_{j}^{m-1}}\\
~x_{0} \wedge x_{i} = 0~\forall~i\in[k]
\end{aligned}
\right\}
\implies x_i = 0,~\forall i \in [k] \cup \{0\}.
\end{equation*}
\item\label{P} extended horizontal ${P}$ ($\mathbb{EHP}$) tensor if 
\begin{equation*}
\left.
\begin{aligned}
\mathcal{A}_{0}x_{0}^{m-1} = {\sum_{j =1}^k \mathcal{A}_{j}x_{j}^{m-1}}
\\
x_{0} \ast x_{i} \leq {0}~\forall~i \in [k] 
\end{aligned}
\right\}
\implies x_i = 0,~\forall i \in [k] \cup \{0\}.
\end{equation*}
\item \label{strong P}  strong  extended horizontal ${P}$ (strong $\mathbb{EHP}$) tensor if for $x_{i}$, $\overline{x}_{i} \in \mathbb{R}^{n}$ where $0 \leq i \leq k,$
\begin{equation*}
\left.
\begin{aligned}
(\mathcal{A}_{0}x_{0}^{m-1}- \mathcal{A}_{0}\overline{x}_0^{m-1}) = {\sum_{j =1}^k (\mathcal{A}_{j}x_j^{m-1}}-\mathcal{A}_{j}\overline{x}_j^{m-1})\\
(x_{0}- \overline{x}_{0}) \ast (x_{i}-\overline{x}_{i}) \leq {0}~\forall~i \in [k]
\end{aligned}
\right\}
\left.
\implies
\begin{aligned} 
 x_{i} = \overline{x}_{i},~\forall i \in [k] \cup \{0\}. 
 \end{aligned}
 \right.
\end{equation*}
\item \label{E} extended horizontal ${E}$ ($\mathbb{EHE}$) tensor if 
\begin{equation*}
\left.
\begin{aligned}
\mathcal{A}_{0}x_{0}^{m-1} = {\sum_{j =1}^k \mathcal{A}_{j}x_{j}^{m-1}}\\
~x_{i} \geq {0}~\forall~i \in [k]\\
x_{0} \ast x_{i} \leq {0}~\forall~i \in [k]
\end{aligned}
\right\}
\implies x_i = 0,~\forall i \in [k] \cup \{0\}.
\end{equation*}
\item \label{ND}  extended horizontal non-degenerate ($\mathbb{EHND}$) tensor if 
\begin{equation*}
\left.
\begin{aligned}
    \mathcal{A}_{0}x_{0}^{m-1} = {\sum_{j =1}^k \mathcal{A}_{j}x_{j}^{m-1}}\\
    x_{i} \ast x_{j} = {0}~\forall~0\leq i < j \leq k
\end{aligned}
\right\}
\implies x_i = 0,~\forall i \in [k] \cup \{0\}.
\end{equation*}
\item \label{strong ND} strong extended horizontal non-degenerate (strong $\mathbb{EHND}$) tensor if for  $x_{i}$, $\overline{x}_{i} \in \mathbb{R}^{n}$ where $0 \leq i \leq k,$
\begin{equation*}
\left.
\begin{aligned}
(\mathcal{A}_{0}x_{0}^{m-1}- \mathcal{A}_{0}\overline{x}_0^{m-1}) = {\sum_{j =1}^k (\mathcal{A}_{j}x_j^{m-1}}-\mathcal{A}_{j}\overline{x}_j^{m-1})\\
(x_{i}- \overline{x}_{i}) \ast (x_{j}-\overline{x}_{j}) = {0}~\forall~0\leq i < j \leq k
\end{aligned}
\right\}
\left.
\implies
\begin{aligned} 
 x_{i} = \overline{x}_{i},~\forall i \in [k] \cup \{0\}. 
 \end{aligned}
 \right.
\end{equation*}
\end{enumerate}
\end{definition}

\begin{remark}\rm\label{notation for convenience}
We use the same notation to denote the class of  $\widehat{\mathcal{A}} \in \Theta^{(k+1)}_{(m,n)}$ having a certain property. 
For example, $\mathbb{EHR}_{0}$ denotes the set of all $\widehat{\mathcal{A}} \in \Theta^{(k+1)}_{(m,n)}$  being an $\mathbb{EHR}_{0}$ tensor.  Observe that, in Definition \ref{special tensors}, for $k =1,$ the $\mathbb{EHR}_{0}$ tensor, $\mathbb{EHP}$ tensor, and strong $\mathbb{EHP}$ tensor reduce to ${R}_{0}$ tensor pair \cite[Definition 3.2]{Yadav30082024}, $P$ tensor pair \cite[Definition 3.12]{Yadav30082024}, and strong $P$ tensor pair \cite[Definition 3.20]{Yadav30082024} respectively.    
\end{remark}

\begin{remark}\rm\label{HTCP}
Note that if $\mathcal{A}_{0} = \mathcal{I}, k =1$ and $\widehat{\mathcal{A}} = (\mathcal{I}, \mathcal{A}_{1}) \in \Theta ^{(2)}_{(m,n)}$, then $\widehat{\mathcal{A}}$ being an $\mathbb{EHR}_{0}$ tensor corresponds to $\mathcal{A}_{1}$ being an ${R}_{0}$ tensor \cite{MR3821072}. Similarly $\mathbb{EHP}$ tensor, and strong $\mathbb{EHP}$ tensor correspond to $\mathcal{A}_{1}$ being a $P$ tensor \cite{MR3341670}, and strong $P$ tensor \cite{MR3513266} respectively. Also, $\mathbb{EHE}$ tensor, and (strong $\mathbb{EHND}$) $\mathbb{EHND}$ tensor correspond to $\mathcal{A}_{1}$ being a strictly semipositive tensor \cite{MR3501398}, and (non-degenerate property) non-degenerate tensor \cite{MR4310678} respectively.   
\end{remark}

In what follows, we provide some examples to illustrate the above definition.
\begin{example}\rm\label{EHR11}
Let $\widehat{\mathcal A} = (\mathcal{A}_{0},\mathcal{A}_{1},\mathcal{A}_{2})$ be in $\Theta^{(3)}_{(3,2)}$ such that
$\mathcal{A}_{0} = (a^{0}_{ijk}) \in \mathbb{T}(3,2)$ where $a^{0}_{122}= 1, a^{0}_{211} = 1$ and all other entries are zero, $\mathcal{A}_{1} = (a^{1}_{ijk}) \in \mathbb{T}(3,2)$ with $a^{1}_{122} = -1, a^{1}_{211} = 1$  and all other entries are zero, and $\mathcal{A}_{2} = (a^{2}_{ijk}) \in \mathbb{T}(3,2)$ with $a^{2}_{122} = -1, a^{2}_{211} = 1$  and all other entries are zero. We show that (i) $\widehat{\mathcal{A}} \in \Theta^{(3)}_{(3,2)}$ is an $\mathbb{EHR}_{0}$ tensor and $\mathbb{EHND}$ tensor, (ii) but none of the tensors $\mathcal{A}_{0}, \mathcal{A}_{1}, \mathcal{A}_{2}$ are  $R_{0}$ tensors, and (iii) $\widehat{\mathcal{A}} \in \Theta^{(3)}_{(3,2)}$ is not an $\mathbb{EHE}$ tensor. For any ${\bf w} = (x,y,z) \in \Theta^{(3)}_{2}$, we have
\begin{equation*}
\mathcal{A}_{0}{x ^ 2} = (x_{2}^{2},x_{1}^{2})^{T}, \mathcal{A}_{1}{y^2} = (-y_{2}^{2}, y_{1}^{2})^{T},~\text{and}~ \mathcal{A}_{2}{z ^ 2} = (-z_{2}^{2},z_{1}^{2})^{T}. 
\end{equation*} 
\begin{enumerate}
\item[\rm(i)] Let ${\bf w} = (x,y,z) \in \Theta^{(3)}_{2}$ satisfy 
\begin{equation*}
\begin{aligned}
&\mathcal{A}_{0}{x^2} = \mathcal{A}_{1}{y^2} + \mathcal{A}_{2}{z^2}, {x} \wedge {y} = 0,~\text{and}~{x} \wedge {z} = 0,\\
\implies & x_{2}^{2} + y_{2}^{2}+z_{2}^{2} = 0, x_{1}^{2} = y_{1}^{2}+z_{1}^{2}, {x \geq 0}, {y \geq 0}, {z \geq 0}, {x} \ast {y} = 0,{x} \ast {z} = 0.
\end{aligned}
\end{equation*}
A simple calculation yields ${x} = 0, {y} = 0, {z} =0.$ Hence ${\bf w} = {\bf 0}.$ Thus $\widehat{\mathcal{A}}$ is an $\mathbb{EHR}_{0}$ tensor. Similarly, if  ${\bf w} = (x,y,z) \in \Theta^{(3)}_{2}$ satisfies
\begin{equation*}
\mathcal{A}_{0}{x^2} = \mathcal{A}_{1}{y^2} + \mathcal{A}_{2}{z^2}, {x} \ast {y} = 0,{x} \ast {z} = 0,~\text{and}~{y} \ast {z} = 0,
\end{equation*} 
then ${\bf w} = {\bf 0}$. Thus, $\widehat{\mathcal{A}}$ is an $\mathbb{EHND}$ tensor. 
\item[\rm(ii)] If ${x} = (0,1)^{T}$, then  $x \wedge \mathcal{A}_{0}{x^2} = 0.$ So $\mathcal{A}_{0}$ is not an $R_{0}$ tensor.
Letting ${y} = (1,0)^{T} = {z}$, we get $y \wedge \mathcal{A}_{1}{y^2} = 0= {z} \wedge \mathcal{A}_{2}{z^2}.$ So $\mathcal{A}_{1}$  and $\mathcal{A}_{2}$ are not $R_{0}$ tensors.
\item[\rm(iii)] Note that  ${x} = (-1,0)^{T}, {y} = (1,0)^{T},~\text{and}~{z} = (0,0)^{T}$ satisfy 
\begin{equation}\label{EHE112}
\mathcal{A}_{0}{x^2} = \mathcal{A}_{1}{y^2}+\mathcal{A}_{2}{z^2},
{y} \geq {0}, {z} \geq {0}, {x} \ast {y} \leq {0}, {x} \ast {z} \leq 0.
\end{equation}
Thus,  Eq.(\ref{EHE112}) admits a nonzero solution. Hence $\widehat{\mathcal{A}}$ cannot be an $\mathbb{EHE}$ tensor. 
\end{enumerate}
 \end{example}
 
 \begin{remark}\rm
 It is evident from Example \ref{EHR11} that an $\mathbb{EHR}_{0}$ tensor  $\widehat{\mathcal{A}} \in \Theta^{(k+1)}_{(m,n)}$ need not to be composed of $R_{0}$ tensors. 
 \end{remark}
 
\begin{example}\rm \label{non-degenerate tensor but not EHND}
Let $\widehat{\mathcal{A}} = (\mathcal{A}_{0}, \mathcal{A}_{1},\mathcal{A}_{2}) \in \Theta^{(3)}_{(3,2)}$ where $\mathcal{A}_{0} = (a^{0}_{ijk}) \in \mathbb{T}(3,2)$ such that $a^{0}_{111} = a^{0}_{211} = a^{0}_{222} = 1$ and other entries are zero, $\mathcal{A}_{1} = (a^{1}_{ijk}) \in \mathbb{T}(3,2)$ with
$a^{1}_{111} = a^{1}_{122} = a^{1}_{222} = 1$ and other entries are zero, and  $\mathcal{A}_{2} = (a^{2}_{ijk}) \in \mathbb{T}(3,2)$ having 
$a^{2}_{111} = a^{2}_{122} = a^{2}_{211} = a^{2}_{222}= 1$ and other entries as zero. 
We show that (i) $\mathcal{A}_{0}, \mathcal{A}_{1}, \mathcal{A}_{2}$ are non-degenerate tensors, (ii) but $\widehat{\mathcal{A}}$ is not an  $\mathbb{EHND}$ tensor, and (iii) $\widehat{\mathcal{A}}$ is not an $\mathbb{EHR}_{0}$ tensor. 
For any ${\bf w}= ({x},{y},{z}) \in \Theta^{(3)}_{2}$, we have
$$\mathcal{A}_{0}{x}^{2} = (x_{1}^{2}, x_{1}^{2} + x_{2}^{2})^{T},~\mathcal{A}_{1}{y}^{2} = ( y_{1}^{2}+y_{2}^{2}, y_{2}^{2})^{T}~\text{and}~\mathcal{A}_{2}{z}^{2} = (z_{1}^{2}+z_{2}^{2}, z_{1}^{2}+z_{2}^{2})^{T}.$$
\begin{enumerate}
\item[\rm(i)] Note that ${x} \ast \mathcal{A}_{0}{x}^{2} = 
(x_{1}^{3},
x_{2}(x_{1}^{2}+x_{2}^{2}))^{T}
.$ Thus ${x} \ast \mathcal{A}_{0}{x}^{2} = 0$ implies that ${x} = {0}.$ Hence $\mathcal{A}_{0}$ is a non-degenerate tensor. Also,
$y \ast \mathcal{A}_{1}{y}^{2} = 
(y_{1}(y_{1}^{2}+y_{2}^{2}),
y_{2}^{3})^{T} = (0,0)^{T}$ gives ${y} = {0}.$ Hence $\mathcal{A}_{1}$ is a non-degenerate tensor, and  $z \ast \mathcal{A}_{2}{z}^{2} = (z_{1}(z_{1}^{2}+z_{2}^{2}),
z_{2}(z_{1}^{2}+z_{2}^{2}))^{T}.$ Thus ${z} \ast \mathcal{A}_{2}{z}^{2} = 0$ yields ${z} = {0}.$ Hence $\mathcal{A}_{2}$ is a non-degenerate tensor. Hence all of the tensors $\mathcal{A}_{0}, \mathcal{A}_{1}, \mathcal{A}_{2}$ are non-degenerate. 
\item[\rm(ii)] Let ${\bf w}= ({x},{y},{z}) = ((1,0)^{T},(0,1)^{T},(0,0)^{T}) \in \Theta^{(3)}_{2}$. Then, we can easily see that the nonzero ${\bf w}$ satisfies
\begin{equation*}
\mathcal{A}_{0}{x}^{2} = \mathcal{A}_{1}{y}^{2} + \mathcal{A}_{2}{z}^{2}, {x} \ast {y} = 0,~ {x} \ast {z} = 0,~\text{and}~{y}\ast {z} = 0 .
\end{equation*}
Thus  $\widehat{\mathcal{A}}$ is not an $\mathbb{EHND}$ tensor.
\item[\rm(iii)]  Let ${\bf w}= ({x},{y},{z}) = ((1,0)^{T},(0,1)^{T},(0,0)^{T}) \in \Theta^{(3)}_{2}$. Then, we can easily see that the nonzero ${\bf w}$ satisfies
\begin{equation*}
\mathcal{A}_{0}{x}^{2} = \mathcal{A}_{1}{y}^{2} + \mathcal{A}_{2}{z}^{2}, {x} \wedge {y} = 0,~ {x} \wedge {z} = 0.
\end{equation*}
Thus  $\widehat{\mathcal{A}}$ is not an $\mathbb{EHR}_{0}$ tensor.
\end{enumerate}
\end{example}

\begin{example}\rm \label{EHND but not non-degenerate}
Let $\widehat{\mathcal{A}} = (\mathcal{A}_{0}, \mathcal{A}_{1},\mathcal{A}_{2}) \in \Theta^{(3)}_{(4,2)}$ where $\mathcal{A}_{0} = (a^{0}_{ijkl}) \in \mathbb{T}(4,2)$ such that $a^{0}_{1222} =1,~a^{0}_{2111} =1$ and other entries are zero,  $\mathcal{A}_{1} = (a^{1}_{ijkl}) \in \mathbb{T}(4,2)$ with
$a^{1}_{1222} =-1,~a^{1}_{2111} = 1$ and other entries are zero, and $\mathcal{A}_{2} = (a^{2}_{ijkl}) \in \mathbb{T}(4,2)$ having
$a^{2}_{1222} =1,~a^{2}_{2111} = -1$ and other entries as zero.
We show that (i) $\widehat{\mathcal{A}}$ is an $\mathbb{EHND}$ tensor and strong $\mathbb{EHND}$ tensor, (ii) but none of the tensors $\mathcal{A}_{0}, \mathcal{A}_{1}, \mathcal{A}_{2}$ are non-degenerate, and (iii) $\widehat{\mathcal{A}}$ is not an  $\mathbb{EHR}_{0}$ tensor. For any ${\bf w}= ({x},{y},{z}) \in \Theta^{(3)}_{2}$, we have
$$\mathcal{A}_{0}{x}^{3} = (x_{2}^{3}, x_{1}^{3})^{T},~\mathcal{A}_{1}{y}^{3} = (-y_{2}^{3}, y_{1}^{3})^{T}~\text{and}~\mathcal{A}_{2}{z}^{3} = (z_{2}^{3}, -z_{1}^{3})^{T}.$$
\begin{enumerate}
\item[\rm(i)] Let ${\bf w}= ({x},{y},{z}) \in \Theta^{(3)}_{2}$ satisfy
\begin{equation*}
\mathcal{A}_{0}{x}^{3} = \mathcal{A}_{1}{y}^{3} + \mathcal{A}_{2}{z}^{3}, {x} \ast {y} = 0,~ {x} \ast {z} = 0,~\text{and}~{y}\ast {z} = 0 .
\end{equation*}
From the above equations, we get
\begin{eqnarray*}
& x_{2}^{3}+y_{2}^{3} - z_{2}^{3} =0, x_{1}^{3}-y_{1}^{3} + z_{1}^{3} = 0,\\
& x_{1}y_{1} = 0, x_{2}y_{2} = 0, x_{1}z_{1} = 0, x_{2}z_{2} = 0, y_{1}z_{1} =0, y_{2}z_{2}=0. 
\end{eqnarray*}
A simple calculation yields ${x}= 0, {y} = {0}, {z} = {0}$, and hence ${\bf w} = {\bf 0}.$ Thus  $\widehat{\mathcal{A}}$ is an  $\mathbb{EHND}$ tensor. Now, for any ${x},\bar{x}, {y},\bar{y},{z},\bar{z} \in \mathbb{R}^{2},$ we have
\begin{equation*}
\mathcal{A}_{0}{x}^{3}-\mathcal{A}_{0}\bar{x}^{3} = 
(x_{2}^{3}-{\bar{x}}_{2}^{3},
x_{1}^{3}-{\bar{x}}_{1}^{3})^{T},~\mathcal{A}_{1}{y}^{3}-\mathcal{A}_{1}{\bar{y}}^{3} = (-y_{2}^{3}+{\bar{y}}_{2}^{3},
y_{1}^{3}-{\bar{y}}_{1}^{3})^{T},~\text{and}
\end{equation*}
\begin{equation*}
\mathcal{A}_{2}{z}^{3}-\mathcal{A}_{2}{\bar{z}}^{3} =
(z_{2}^{3} -{\bar{z}}_{2}^{3},-z_{1}^{3}+{\bar{z}}_{1}^{3})^{T}.
\end{equation*}
Suppose that ${x},\bar{x}, {y},\bar{y},{z},\bar{z} \in \mathbb{R}^{2}$  satisfy
\begin{equation*}
\begin{aligned}
&\mathcal{A}_{0}{x}^{3}-\mathcal{A}_{0}\bar{x}^{3} = (\mathcal{A}_{1}{y}^{3}-\mathcal{A}_{1}{\bar{y}}^{3}) + (\mathcal{A}_{2}{z}^{3}-\mathcal{A}_{2}{\bar{z}}^{3}),~\text{and}\\
& ({x}-\bar{x}) \ast ({y}-\bar{y}) = 0, ({x} -\bar{x})\ast({z} -\bar{z}) = {0}~\text{and}~({y}-\bar{y})\ast({z}-\bar{z})=0.
\end{aligned}
\end{equation*}
A simple calculation yields that ${x} =\bar{x}, {y} =\bar{y}$ and ${z} =\bar{z}.$ Hence $\widehat{\mathcal{A}}$ is a strong $\mathbb{EHND}$ tensor.
\item[\rm(ii)] If $x= (1,0)^{T} = y$, then $x \ast \mathcal{A}_{0}{x}^{3} = 0 = y \ast \mathcal{A}_{1}{y}^{3}$. Thus $\mathcal{A}_{0}$ and $\mathcal{A}_{1}$ are  not non-degenerate. Taking $z= (0,1)^{T}$ gives  $z \ast \mathcal{A}_{2}{z}^{3} = 0.$ Thus $\mathcal{A}_{2}$ is not non-degenerate. Hence none of the tensors $\mathcal{A}_{0}, \mathcal{A}_{1}, \mathcal{A}_{2}$ are non-degenerate.
\item[\rm(iii)] Let ${\bf w} = (x,y,z) = ((0,0)^{T}, (1,1)^{T}, (1,1)^{T}) \in \Theta^{(3)}_{2}$. Then we can see easily that the nonzero ${\bf w}$ satisfies
\begin{equation*}
\mathcal{A}_{0}{x ^{3}} = \mathcal{A}_{1}{y^{3}} + \mathcal{A}_{2}{z^{3}},~{x}\wedge {y} = 0~\text{and}~{x} \wedge {z} = {0}.
\end{equation*}
Hence $\widehat{\mathcal{A}}$ is not an  $\mathbb{EHR}_{0}$ tensor.
\end{enumerate}
\end{example}

\begin{remark}\rm
We would like to mention that
\begin{enumerate}
  \item[\rm(i)] even if all the tensors $\mathcal{A}_{i} \in \mathbb{T}(m,n)$, where $0 \leq i \leq k,$ are non-degenerate tensors, $\widehat{\mathcal{A}} = (\mathcal{A}_{0}, \mathcal{A}_{1},...,\mathcal{A}_{k}) \in \Theta^{(k+1)}_{(m,n)}$ need not necessarily  an $\mathbb{EHND}$ tensor (see Example \ref{non-degenerate tensor but not EHND}).
\item[\rm(ii)] $\widehat{\mathcal{A}} = (\mathcal{A}_{0}, \mathcal{A}_{1},...,\mathcal{A}_{k}) \in \Theta^{(k+1)}_{(m,n)}$ being an $\mathbb{EHND}$ tensor need not imply that all the tensors $\mathcal{A}_{i} \in \mathbb{T}(m,n)$, where $0 \leq i \leq k,$ are  non-degenerate tensors (refer to Example \ref{EHND but not non-degenerate}).
\end{enumerate}
\end{remark}

\begin{example}\rm
Let $\widehat{\mathcal{A}} = (\mathcal{A}_{0}, \mathcal{A}_{1}) \in  \Theta^{(2)}_{(4,2)}$ where $\mathcal{A}_{0} = (a^{0}_{ijkl}) \in \mathbb{T}(4,2)$ such that $a^{0}_{1111} =1, a^{0}_{2222} =1$ and other entries are zero, and $\mathcal{A}_{1} = (a^{1}_{ijkl}) \in \mathbb{T}(4,2)$ with $a^{1}_{1111} =1, a^{1}_{1222} =1, a^{1}_{2222} =1$ and other entries are zero. We show that $\widehat{\mathcal{A}}$ is a strong $\mathbb{EHP}$ tensor. For any $x, \bar{x}, y, \bar{y}$ in $\mathbb{R}^{2}$, we have
$$\mathcal{A}_{0}{x^3} - \mathcal{A}_{0}{\bar{x}^3} = \begin{bmatrix}
x_{1}^{3} - \bar{x}_{1}^{3}\\
x_{2}^{3} - \bar{x}_{2}^{3}
\end{bmatrix}
,~\text{and}~\mathcal{A}_{1}{y^3} -\mathcal{A}_{1}{\bar{y}^{3}} = \begin{bmatrix}
y_{1}^{3}+y_{2}^{3} - (\bar{y}_{1}^{3}+\bar{y}_{2}^{3})\\
y_{2}^{3} -\bar{y}_{2}^{3}
\end{bmatrix}.$$ 
Suppose that ${x},\bar{x}, {y},\bar{y} \in \mathbb{R}^{2}$  satisfy
\begin{equation*}
\mathcal{A}_{0}{x}^{3}-\mathcal{A}_{0}\bar{x}^{3} = \mathcal{A}_{1}{y}^{3}-\mathcal{A}_{1}{\bar{y}}^{3},~\text{and}~
 ({x}-\bar{x}) \ast ({y}-\bar{y}) \leq  0.
\end{equation*}
A simple calculation yields ${x} = \bar{x}$ and ${y} = \bar{y}$. 
Hence $\widehat{\mathcal{A}}$ is a strong $\mathbb{EHP}$ tensor.
\end{example}

\begin{proposition}\label{subset}
The following statements are valid.
\begin{enumerate}
\item[\rm(i)] $\text{strong}~ \mathbb{EHP}   \subseteq  \mathbb{EHP} \subseteq \mathbb{EHE} \subseteq \mathbb{EHR}_{0},$ 
\item[\rm(ii)] $\mathbb{EHND} \subseteq \mathbb{EHR}_{0}$, when $0\leq k \leq 1.$ This need not be true in general (see Example \ref{EHND but not non-degenerate} for $k = 2$),
\item[\rm(iii)] $\mathbb{EHP}  \subseteq \mathbb{EHND}$ and strong $\mathbb{EHP} \subseteq \text{strong}~\mathbb{EHND}\subseteq  \mathbb{EHND},$    
\item[\rm(iv)] $ \text{strong}~ \mathbb{EHND} \subseteq  \mathbb{EHR}_{0},$ when $0 \leq k \leq 1$. This need not be true in general (see Example \ref{EHND but not non-degenerate} for $k = 2$).
\end{enumerate}
\begin{proof}
\begin{enumerate}
\item[\rm(i)] This follows  from the Definition \ref{special tensors}.
\item[\rm(ii)] When $k = 0$, it follows from \cite[Proposition 4.1]{MR4310678} that a non-degenerate tensor is an $R_{0}$ tensor. Now, suppose that $k =1$ and $\widehat{\mathcal{A}} = (\mathcal{A}_{0}, \mathcal{A}_{1}) \in \Theta^{(2)}_{(m,n)}$ is an  $\mathbb{EHND}$ tensor. This gives
\begin{equation*}
\mathcal{A}_{0}{x_{0}^{m-1}} = \mathcal{A}_{1}{x_{1}^{m-1}}~\text{and}~{x_{0}} \ast {x_{1}} = 0.
\end{equation*} 
We claim that $\widehat{\mathcal{A}}$ is an $\mathbb{EHR}_{0}$ tensor. Let ${\bf x} = (x_{0},x_{1}) \in \Theta^{(2)}_{n}$ satisfy
\begin{equation*}
\mathcal{A}_{0}{x_{0}^{m-1}} = \mathcal{A}_{1}{x_{1}^{m-1}}~\text{and}~{x_{0}} \wedge {x_{1}} = 0.
\end{equation*}
This implies that $x_{0} \geq {0}, x_{1} \geq {0}~\text{and}~x_{0} \ast x_{1} = 0.$ As $\widehat{\mathcal{A}}$ is an  $\mathbb{EHND}$ tensor, we get ${\bf x} = {\bf 0}.$ Hence $\widehat{\mathcal{A}}$ is an  $\mathbb{EHR}_{0}$ tensor. 
\item[\rm(iii)] Let $\widehat{\mathcal{A}} = (\mathcal{A}_{0}, \mathcal{A}_{1},...,\mathcal{A}_{k}) \in \Theta^{(k+1)}_{(m,n)}$ be an $\mathbb{EHP}$ tensor. Suppose that ${\bf x} = (x_{0},x_{1},...,x_{k}) \in \Theta^{(k+1)}_{n}$ satisfies
\begin{equation*}
\begin{aligned}
    \mathcal{A}_{0}x_{0}^{m-1} = {\sum_{j =1}^k \mathcal{A}_{j}x_{j}^{m-1}}~\text{and}~
    x_{i} \ast x_{j} = {0}~\forall~0\leq i < j \leq k.
\end{aligned}
\end{equation*}
In particular, ${\bf x} = (x_{0},x_{1},...,x_{k}) \in \Theta^{(k+1)}_{n}$ satisfies
\begin{equation*}
\mathcal{A}_{0}x_{0}^{m-1} = {\sum_{j =1}^k \mathcal{A}_{j}x_{j}^{m-1}}~\text{and}~
x_{0} \ast x_{i} \leq {0}~\forall~i \in [k]. 
\end{equation*}
As $\widehat{\mathcal{A}}$ is an  $\mathbb{EHP}$ tensor, we get ${\bf x} = {\bf 0}$ and thus $\widehat{\mathcal{A}}$ is an  $\mathbb{EHND}$ tensor. In a similar manner we can show that if  $\widehat{\mathcal{A}} = (\mathcal{A}_{0}, \mathcal{A}_{1},...,\mathcal{A}_{k}) \in \Theta^{(k+1)}_{(m,n)}$ is a strong $\mathbb{EHP}$ tensor, then $\widehat{\mathcal{A}}$ is a strong $\mathbb{EHND}$ tensor.     
Now assume that $\widehat{\mathcal{A}} = (\mathcal{A}_{0}, \mathcal{A}_{1},...,\mathcal{A}_{k}) \in \Theta^{(k+1)}_{(m,n)}$ is a strong $\mathbb{EHND}$ tensor. We show that $\widehat{\mathcal{A}}$ is an $\mathbb{EHND}$ tensor. Let ${{\bf x}} = (x_{0},x_{1},...,x_{k}) \in \Theta^{(k+1)}_{n}$ satisfy
\begin{equation}\label{EHND}
\mathcal{A}_{0}x_{0}^{m-1} = {\sum_{j =1}^k \mathcal{A}_{j}x_{j}^{m-1}}~\text{and}~x_{i} \ast x_{j} = {0}~\forall~0\leq i < j \leq k.
\end{equation}
In particular, Eq.(\ref{EHND}) is satisfied for all $\overline{{\bf x}} = ({\bar x}_{0}, \bar {x}_{1},...,\bar {x}_{k}) \in \Theta^{(k+1)}_{n}  $ such that $\bar{x}_{i} = 0,$ for all $0 \leq i \leq k.$ Since $\widehat{\mathcal{A}}$ is a strong $\mathbb{EHND}$ tensor, we get ${\bf x} = {\bf 0}$ and therefore $\widehat{\mathcal{A}}$ is an  $\mathbb{EHND}$ tensor.
\item[\rm(iv)] From  (ii) and (iii), it is clear that $\text{strong}~ \mathbb{EHND}\subseteq \mathbb{EHR}_{0},\text{for}~ k \in \{0,1\}.$   
\end{enumerate}
\end{proof}
\end{proposition}

\section{Properties of the Solution Set of the EHTCP}\label{Section 4}
In this section, we first define the degree of the EHTCP and then we provide several properties of the solution set of the EHTCP such as  nonemptiness, compactness, uniqueness and finiteness with regards to the newly defined  structured tensors.

\subsection{Degree of EHTCP}\label{Degree of EHTCP}
Let $\widehat{\mathcal{A}} = (\mathcal{A}_{0}, \mathcal{A}_{1},...,\mathcal{A}_{k}) \in \Theta^{(k+1)}_{(m,n)},$~ $\widehat{{\bf d}} = (d_{1},d_{2},...,d_{k-1}) \in \Theta^{(k-1)}_{n,++}$ and $q \in \mathbb{R}^{n}$. Corresponding to the ${\rm EHTCP}(\widehat{\mathcal{A}},\widehat{\bf d},q)$, we define two functions $\Psi_{(\widehat{\mathcal {A}},\widehat{\bf d})}$ and $\Psi_{\widehat{\mathcal{A}}}$ on $\Theta^{(k+1)}_{n}$ as follows:
\begin{equation*}
\Psi_{(\widehat{\mathcal {A}},\widehat{\bf d})} ({\bf x}) = \begin{bmatrix}
x_{0} \wedge x_{1} \\
(d_{1} - x_{1}) \wedge x_{2}\\
\vdots\\
(d_{k-1} - x_{k-1}) \wedge x_{k}\\
\mathcal{A}_{0} {x_{0}^{m-1}} - {\displaystyle\sum_{j = 1}^{k} \mathcal{A}_{j}{x_{j}^{m-1}}}
\end{bmatrix} 
\text{and}~
\Psi_{\widehat{\mathcal{A}}} ({\bf x}) =  
\begin{bmatrix}
x_{0} \wedge x_{1} \\
x_{0} \wedge x_{2}\\
\vdots\\
x_{0} \wedge x_{k}\\
\mathcal{A}_{0} {x_{0}^{m-1}} - {\displaystyle\sum_{j = 1}^{k} \mathcal{A}_{j}{x_{j}^{m-1}}}
\end{bmatrix}. 
\end{equation*}
Note that if  $\widehat{\mathcal{A}}$ is an $\mathbb{EHR}_{0}$ tensor, then $\Psi_{\widehat{\mathcal{A}}}({\bf x}) = {\bf 0} \iff {\bf x} = {\bf 0}.$ Let $\Omega \subseteq \Theta^{(k+1)}_{n}$ be any bounded open set containing the zero vector.   Then the degree of the function $\Psi_{\widehat{\mathcal{A}}}$ over $\Omega$  with respect to zero is defined and it is independent of $\Omega$. We denote $\deg(\Psi_{\widehat{\mathcal{A}}},\Omega,{\bf 0})$ as $\deg(\widehat{\mathcal{A}})$ and call this as the $\rm{EHTCP}$-degree. This $\rm{EHTCP}$-degree coincides with the $\rm{HTCP}$-degree \cite{Yadav30082024}, when $k =1$ and with the $\rm{TCP}$-degree when $k =1$ and $\mathcal{A}_{0} = \mathcal{I}$.

With regards to Lemma \ref{solution}, one can observe that $\Psi_{(\widehat{\mathcal {A}},{\widehat{\bf d}})} ({\bf x}) = {\bf 0} \implies \Psi_{\widehat{\mathcal{A}}}({\bf x}) = {\bf 0}$. Hence, if $\widehat{\mathcal{A}}$ is an $\mathbb{EHR}_{0}$ tensor, then $\Psi_{(\widehat{\mathcal {A}},{\widehat{\bf d}})} ({\bf x}) = {\bf 0} \iff {\bf x} = {\bf 0}.$  Thus for any bounded open set $\Omega \subseteq \Theta^{(k+1)}_{n}$ containing the zero vector, we have 
$$\deg(\Psi_{(\widehat{\mathcal {A}},{\widehat{\bf d}})}, \Omega, {\bf 0}) = \deg(\Psi_{(\widehat{\mathcal {A}},{\widehat{\bf d}})},{\bf 0}).$$

In the following lemma, we show that $\deg(\widehat{\mathcal A})$ (the EHTCP-degree) is same as $\deg(\Psi_{(\widehat{\mathcal{A}}, {\widehat{\bf d}})}, {\bf 0})$ in the presence of an $\mathbb {EHR}_{0}$ tensor.

\begin{lemma}\label{degsame}
Let $\widehat{\mathcal{A}} = (\mathcal{A}_{0}, \mathcal{A}_{1},...,\mathcal{A}_{k}) \in \Theta^{(k+1)}_{(m,n)},$~ $\widehat{{\bf d}} = (d_{1},d_{2},...,d_{k-1}) \in \Theta^{(k-1)}_{n,++}$. If $\widehat{\mathcal{A}}$ is an $\mathbb{EHR}_{0}$ tensor, then
$$\deg(\Psi_{(\widehat{\mathcal{A}}, {\widehat{\bf d}})}, {\bf 0}) = \deg(\widehat{\mathcal{A}}).$$
\end{lemma}
\begin{proof}
Let us consider a homotopy $\Phi : \Theta^{(k+1)}_{n} \times [0,1] \to  \Theta^{(k+1)}_{n}$ defined as
\begin{equation}\label{degree is same}
\Phi ({\bf x},t) = \begin{bmatrix}
x_{0} \wedge x_{1} \\
(t(d_{1} - x_{1}) + (1-t)x_{0}) \wedge x_{2}\\
\vdots\\
(t(d_{k-1} - x_{k-1}) + (1-t)x_{0}) \wedge x_{k}\\
\mathcal{A}_{0}x_{0}^{m-1} - {\displaystyle\sum_{j = 1}^{k} \mathcal{A}_{j}{x_{j}^{m-1}}}
\end{bmatrix},
\end{equation}
where ${\bf x} = (x_{0},x_{1},...,x_{k}) \in \Theta^{(k+1)}_{n},~ \widehat{{\bf d}} = (d_{1},d_{2},...,d_{k-1}) \in \Theta^{(k-1)}_{n,++}~~\text{and}~~\widehat{\mathcal{A}} = (\mathcal{A}_{0}, \mathcal{A}_{1},...,\mathcal{A}_{k}) \in \Theta^{(k+1)}_{(m,n)}$. For $t = 0$, we have $\Phi({\bf x}, {0}) = \Psi_{\widehat{\mathcal{A}}}({\bf x})$ and when $t =1$, we have $\Phi({\bf x},{1}) = \Psi_{(\widehat{\mathcal{A}},\widehat{\bf d})}({\bf x}).$ We show that the zero set of $\Phi({\bf x},t)$ contains only the zero vector. Assume that for some $t \in [0,1],$ $\Phi({\bf x},t) = {\bf 0}.$ From the first row in the vector of  Eq. (\ref{degree is same}), we get $x_{0} \wedge x_{1} = {0}.$ This implies that there exists an index set $I \subseteq [n]$ such that
\begin{equation}\label{alpha}
(x_{0})_{i} = \begin{cases} > 0,  & i \in I \\
0, & i \notin I
 \end{cases}
 ~\text{and}~(x_1)_i=0 ~\forall i\in I.
\end{equation} 
From the second row in the vector of Eq.(\ref{degree is same}), we have
$$(t(d_{1} - x_{1}) + (1-t)x_{0}) \wedge x_{2} = 0.$$
This gives $x_{2} \geq {0}$. As $d_{1} > 0$ and $(x_{0})_{i} > {0}$ for all $i \in I$ implies that $(x_{2})_{i} = 0$ for all $i \in I$. So from Eq.(\ref{alpha}), it can be seen easily that $x_{0} \wedge x_{2} = 0.$ By a  similar process, we get $x_{0} \wedge x_{j} = 0,~\forall~j \in [k].$ Thus $\Phi({\bf x},t) = {\bf 0}$ gives
$$\mathcal{A}_{0}x_{0}^{m-1} = {\displaystyle\sum_{j = 1}^{k} \mathcal{A}_{j}{x_{j}^{m-1}}}~\text{and}~x_{0} \wedge x_{j} = 0,~\forall~j \in [k].$$ 
As $\widehat{\mathcal{A}}$ is an $\mathbb{EHR}_{0}$ tensor, we get ${\bf x} = {\bf 0}.$ Hence $\Phi({\bf x},t) = {\bf 0} \iff {\bf x} = {\bf 0}.$ So by the property 3 (homotopy invariance), for any bounded open set $\Omega$ containing zero vector in $\Theta^{(k+1)}_{n}$, we have
$$\deg(\Psi_{(\widehat{\mathcal{A}},\widehat{\bf d})}, \Omega, {\bf 0}) = \deg(\Psi_{\widehat{\mathcal{A}}}, \Omega, {\bf 0}) \implies \deg(\Psi_{(\widehat{\mathcal{A}},\widehat{\bf d})}, {\bf 0}) = \deg(\widehat{\mathcal{A}}),$$
due to $\widehat{\mathcal{A}}$ being an $\mathbb{EHR}_{0}$ tensor.
\end{proof}

\subsection{Existence Results}\label{Existence and Compactness}
In this subsection, we establish the nonemptiness and compactness of ${\rm SOL}(\widehat{\mathcal{A}},{\widehat{\bf d}},{q})$. Prior to proving our first existence result for the EHTCP, we present a theorem concerning the boundedness of the solution set of EHTCP. 
We omit the proof as it is identical to the one in the classic  case of TCP \cite[Theorem 3.2]{MR3513267}. 
 Note that ${\rm SOL}(\widehat{\mathcal{A}},{\widehat{\bf d}},{q})$ can be an empty set for some ${q} \in {\mathbb R}^{n}$ and $\widehat{\bf d} \in \Theta^{(k-1)}_{n,++}$. 
\begin{theorem}\label{bounded}
Let $\widehat{\mathcal{A}} = (\mathcal{A}_{0}, \mathcal{A}_{1},...,\mathcal{A}_{k}) \in \Theta^{(k+1)}_{(m,n)}.$ If $\widehat{\mathcal{A}}$ is an  $\mathbb{EHR}_{0}$ tensor, then ${\rm SOL}(\widehat{\mathcal{A}}, \widehat{\bf d},{q})$ is bounded for any $\widehat{\bf d} \in \Theta^{(k-1)}_{n,++}$ and ${q} \in {\mathbb R}^{n}$.
\end{theorem}
The converse of the above theorem is not valid even in the case of $m =2$ (see, \cite[Example 3.1]{yadav2023generalizations}).
We now provide our first existence result for EHTCP.
\begin{theorem}\label{nonempty solution}
Let $\widehat{\mathcal{A}} = (\mathcal{A}_{0}, \mathcal{A}_{1},...,\mathcal{A}_{k}) \in \Theta^{(k+1)}_{(m,n)}$ be an $\mathbb{EHR}_{0}$ tensor and $\deg({\widehat{\mathcal{A}}}) \neq 0$. Then ${\rm SOL}(\widehat{\mathcal{A}}, \widehat{\bf d},{q})$ is nonempty and compact for each $\widehat{\bf d} \in \Theta^{(k-1)}_{n,++}$ and ${q} \in {\mathbb R}^{n}$.
\end{theorem}
\begin{proof}
Let $\widehat{\mathcal{A}} = (\mathcal{A}_{0}, \mathcal{A}_{1},...,\mathcal{A}_{k}) \in \Theta^{(k+1)}_{(m,n)},~~ \widehat{\bf d} = (d_{1},d_{2},...,d_{k-1}) \in \Theta^{(k-1)}_{n,++}~\text{and}~{q} \in {\mathbb R}^{n}$. In view of Theorem \ref{bounded}, it is enough to show that ${\rm SOL}(\widehat{\mathcal{A}}, \widehat{\bf d},{q})$ is nonempty, as ${\rm SOL}(\widehat{\mathcal{A}}, \widehat{\bf d},{q})$ is always a closed set. Let us define a homotopy $\Phi : \Theta^{(k+1)}_{n} \times [0,1] \to \Theta^{(k+1)}_{n}$  as
\begin{equation*}
\Phi({\bf x},t) = \begin{bmatrix}
x_{0} \wedge x_{1}\\
(d_{1} - x_{1})\wedge x_{2}\\
\vdots\\
(d_{k-1} - x_{k-1}) \wedge x_{k}\\
 \mathcal{A}_{0}x_{0}^{m-1} - {\displaystyle\sum_{j = 1}^{k} \mathcal{A}_{j}{x_{j}^{m-1}}} - tq
\end{bmatrix}
\end{equation*}  
Let $\widehat{\bf q} := (0,0,...0,q) \in  \Theta^{(k+1)}_{n}.$ Then,
$$\Phi({\bf x}, 0) = \Psi_{(\widehat{\mathcal{A}}, {\widehat{\bf d}})}({\bf x})~\text{and}~\Phi({\bf x}, 1) = \Psi_{(\widehat{\mathcal{A}}, {\widehat{\bf d}})}({\bf x}) - {\widehat{\bf q}}. $$
By using the similar argument as in the proof of Lemma \ref{degsame}
, we can show easily that the zero set of the homotopy $\Phi({\bf x},t)$, say $X$, is bounded. Hence by the property 3 (homotopy invariance), we get $\deg(\Psi_{(\widehat{\mathcal{A}}, {\widehat{\bf d}})},  \Omega, {\bf 0}) = \deg(\Psi_{(\widehat{\mathcal{A}}, {\widehat{\bf d}})} - {\widehat{\bf q}}, \Omega, {\bf 0}),$ for any bounded open set $\Omega$ containing $X$ in $\Theta^{(k+1)}_{n}$. Using Lemma \ref{degsame} and $\deg(\widehat{\mathcal{A}}) \neq 0$, we get $\deg(\Psi_{(\widehat{\mathcal{A}}, {\widehat{\bf d}})} - {\widehat{\bf q}}, \Omega, {\bf 0}) \neq 0$ implying that  ${\rm SOL}(\widehat{\mathcal{A}}, \widehat{\bf d},{q})$ is nonempty.
\end{proof}

The following theorem gives the nonemptiness and compactness of ${\rm SOL}(\widehat{\mathcal{A}},\widehat{{\bf d}},{q})$ when $\widehat{\mathcal{A}}$ is an $\mathbb{EHE}$ tensor.
\begin{theorem}\label{E imply nonempty}
Let $\widehat{\mathcal{A}} = (\mathcal{A}_{0}, \mathcal{A}_{1},...,\mathcal{A}_{k}) \in \Theta^{(k+1)}_{(m,n)},$ where $m$ is even. If $\widehat{\mathcal{A}}$ is an $\mathbb{EHE}$ tensor, then ${\rm SOL}(\widehat{\mathcal{A}},\widehat{{\bf d}},{q})$ is  nonempty and compact for each $\widehat{\bf d} \in \Theta^{(k-1)}_{n,++}$ and ${q} \in {\mathbb R}^{n}$. 
\end{theorem}
\begin{proof}
Let $m$ be even and $\widehat{\mathcal{A}} = (\mathcal{A}_{0}, \mathcal{A}_{1},...,\mathcal{A}_{k}) \in \Theta^{(k+1)}_{(m,n)}$ be an $\mathbb{EHE}$ tensor. From  Proposition \ref{subset} (i), it follows that $\widehat{\mathcal{A}}$ is an  $\mathbb{EHR}_{0}$ tensor. In view of Theorem \ref{nonempty solution}, it is sufficient to show that $\deg(\widehat{\mathcal{A}}) \neq 0.$ Let us define a homotopy $\Phi : \Theta^{(k+1)}_{n} \times [0,1] \to \Theta^{(k+1)}_{n}$ as
$$\Phi({\bf x},t)=~ t \begin{bmatrix}
x_{1}\\
x_{2}\\
\vdots\\
x_{k}\\
\mathcal{A}_{0}x_{0}^{m-1}
\end{bmatrix} + (1-t) \begin{bmatrix}
x_{0} \wedge x_{1}\\
x_{0} \wedge x_{2}\\
\vdots\\
x_{0} \wedge x_{k}\\
\mathcal{A}_{0}x_{0}^{m-1} - {\displaystyle\sum_{j = 1}^{k} \mathcal{A}_{j}{x_{j}^{m-1}}}
\end{bmatrix},$$
for any ${\bf x} = (x_{0},x_{1},...,x_{k}) \in \Theta^{(k+1)}_{n}.$ For $t = 0$ and $t =1$, we have  
$$\Phi({\bf x},{0}) = \Psi_{\widehat{\mathcal{A}}}({\bf x})~\text{and}~
\Phi({\bf x},{1}) = \begin{bmatrix}
x_{1}\\
x_{2}\\
\vdots\\
x_{k}\\
\mathcal{A}_{0}x_{0}^{m-1}
\end{bmatrix}.$$
We first show that the set $X = \{{\bf x} \in \Theta^{(k+1)}_{n}: \Phi({\bf x},{t}) = {\bf 0}~\text{for some}~ t \in [0,1]\}$ contains only the zero vector.
To show this, we consider the following cases:

{Case 1.} When $t =0$, we have $\Phi({\bf x},{0}) = \Psi_{\widehat{\mathcal{A}}}({\bf x})$. As $\widehat{\mathcal{A}}$  is an $\mathbb{EHR}_{0}$ tensor, we get $\Psi_{\widehat{\mathcal{A}}}({\bf x}) = {\bf 0} \iff {\bf x} = {\bf 0}.$ Also, for $t =1$, $\Phi({\bf x},{1}) = {\bf 0}$ implies $x_{i} = 0,~\text{for all}~i \in [k]$ and $\mathcal{A}_{0}x_{0}^{m-1} = 0.$ As $\widehat{\mathcal{A}}$ is an  $\mathbb{EHE}$ tensor, this gives ${\bf x} = {\bf 0}.$ Hence, we have $\Phi({\bf x},{1}) = {\bf 0} \iff {\bf x} = {\bf 0}$.

{Case 2.} Let $t \in (0,1).$ Then
\allowdisplaybreaks
\begin{eqnarray}\label{rows}
&\Phi({\bf x},t) = {\bf 0} \notag\\
\implies & \begin{bmatrix}
x_{0} \wedge x_{1}\\
x_{0} \wedge x_{2}\\
\vdots\\
x_{0} \wedge x_{k}\\
\mathcal{A}_{0}x_{0}^{m-1} - {\displaystyle\sum_{j = 1}^{k} \mathcal{A}_{j}{x_{j}^{m-1}}}
\end{bmatrix} = - {\beta} {\begin{bmatrix}
x_{1}\\
x_{2}\\
\vdots\\
x_{k}\\
\mathcal{A}_{0}x_{0}^{m-1}
\end{bmatrix}},
\end{eqnarray} 
 where ${\beta} = \dfrac{t}{1-t} > {0}$. 
 From the first $k$-rows of Eq.(\ref{rows}), we have
 \allowdisplaybreaks 
 \begin{eqnarray}\label{nonpositive}
& x_{0} \wedge x_{i} = -\beta x_{i}, \forall~ i \in [k] \notag\\
\implies & (x_{0} + \beta x_{i}) \wedge ((1+ \beta) x_{i}) = {0},\forall~ i \in [k] \notag \\
\implies & x_{i} \geq {0}~\text{and}~x_{0} \ast x_{i} = - {\beta} x_{i}^{[2]},\forall~ i \in [k] \notag \\
\implies & x_{i} \geq {0}~\text{and}~ x_{0} \ast x_{i} \leq {0},\forall~ i \in [k].
 \end{eqnarray}
From the last row of Eq.(\ref{rows}), we have
 \allowdisplaybreaks 
\begin{eqnarray}\label{last row}
&(1 + \beta)\mathcal{A}_{0}x_{0}^{m-1} = {\displaystyle\sum_{j = 1}^{k} \mathcal{A}_{j}{x_{j}^{m-1}}} \notag \\
\implies  & \mathcal{A}_{0}({\beta}'x_{0})^{m-1} = {\displaystyle\sum_{j = 1}^{k} \mathcal{A}_{j}{x_{j}^{m-1}}} ,
\end{eqnarray}
where ${\beta}^{'} = (1 + \beta)^{\frac{1}{m-1}}$. As $m$ is even, from Eqs. (\ref{nonpositive}) and (\ref{last row}), we have
\begin{equation*}
\mathcal{A}_{0}({\beta}'x_{0})^{m-1} = {\displaystyle\sum_{j = 1}^{k} \mathcal{A}_{j}{x_{j}^{m-1}}},~x_{i}\geq {0}~\text{and}~({\beta}^{'} x_{0}) \ast x_{i} \leq {0},\forall~i \in [k].
\end{equation*}
As $\widehat{\mathcal{A}}$ is an $\mathbb{EHE}$ tensor. From the above equation, we get ${\bf x} = {\bf 0}.$ Hence the set $X$ contains only the zero vector. Let $\Omega$ be any bounded open set containing $X$. Then by the property 3 (homotopy invariance), we get
\begin{equation}\label{EE1}
\deg({\widehat{\mathcal{A}}}) = \deg(\Phi(\cdot, 1), \Omega, {\bf 0}).
\end{equation}
Let $\phi({x_0}) = {\mathcal{A}_{0}}{x_{0}^{m-1}}$ and ${U}_{i}, 1 \leq i \leq k+1$ be arbitrary bounded open sets containing zero such that $\Omega = \displaystyle\prod_{i \in [k+1]}{U}_{i}$. By the Cartesian product property of degree (see \cite[Proposition 2.1.3]{MR1955649}), we have
\begin{equation}\label{EE2}
\deg(\Phi(\cdot, 1), \Omega, {\bf 0}) = \deg(\phi, {U}_{k+1}, {0}) \displaystyle\prod_{i \in [k]} \deg(I, U_{i}, {0}).
\end{equation}
As $\widehat{\mathcal{A}}$ is an $\mathbb{EHE}$ tensor, $\phi({x_0}) = {0} \iff  {{x_0} =0}.$ Since $m$ is even, from \cite[Lemma 3.17]{Yadav30082024}, we have $\deg(\phi, U_{k+1}, {0}) \neq 0$.  From Eqs. (\ref{EE1}) and (\ref{EE2}), we get
\begin{equation*}
\deg(\widehat{\mathcal{A}}) = \deg(\Phi(\cdot,1),\Omega,{\bf 0}) \neq 0.
\end{equation*} 
Hence  ${\rm SOL}(\widehat{\mathcal{A}},\widehat{{\bf d}},{q})$ is  nonempty and compact for each $\widehat{\bf d} \in \Theta^{(k-1)}_{n,++}$ and ${q} \in {\mathbb R}^{n}$. 
\end{proof}

The following example illustrates that the Theorem \ref{E imply nonempty} is not valid in the case of $m$ being odd.
\begin{example}\rm
Let $\widehat{\mathcal{A}} = (\mathcal{A}_{0},\mathcal{A}_{1}) \in \Theta^{(2)}_{(3,2)}$ where $\mathcal{A}_{0} = (a^{0}_{ijk}) \in \mathbb{T}(3,2)$ such that $a^{0}_{111} =1, a^{0}_{122} = 1, a^{0}_{222} = 1$ and other entries are zero, and $\mathcal{A}_{1} = (a^{1}_{ijk}) \in \mathbb{T}(3,2)$ with $a^{1}_{111} =-1, a^{1}_{222} =-1$ and other entries being zero. For any ${x},{y} \in \mathbb{R}^{2}$, we have  $\mathcal{A}_{0}{x^2} = (x_{1}^{2}+x_{2}^{2}, x_{2}^{2})^{T}$ and  $\mathcal{A}_{1}{y^2} = (-y_{1}^{2}, - y_{2}^{2})^{T}.$ Note that
\begin{equation*}
\mathcal{A}_{0}{x^2} = \mathcal{A}_{1}{y^2}, y \geq 0~\text{and}~{x} \ast {y} \leq 0 \implies (x,y)= (0,0).
\end{equation*}
 Therefore $\widehat{\mathcal{A}}$ is an $\mathbb{EHE}$ tensor, but for $q = (-1,-1)^{T}$ in $\mathbb{R}^{2}$, there does not exist any $(x,y) \in \Theta^{(2)}_{2}$ such that
 $\mathcal{A}_{0}{x^2} = {q} + \mathcal{A}_{1}{y^2}~\text{and}~{x} \wedge {y}=0.$
\end{example}

As a consequence of Proposition \ref{subset} and Theorem \ref{E imply nonempty}, we have the following results.
\begin{corollary}\label{P imply nonempty}
Let $\widehat{\mathcal{A}} = (\mathcal{A}_{0}, \mathcal{A}_{1},...,\mathcal{A}_{k}) \in \Theta^{(k+1)}_{(m,n)}$ and $m$ be even. If $\widehat{\mathcal{A}}$ is an $\mathbb{EHP}$ tensor, then ${\rm SOL}(\widehat{\mathcal{A}},\widehat{{\bf d}},{q})$ is  nonempty and compact for each $\widehat{\bf d} \in \Theta^{(k-1)}_{n,++}$ and ${q} \in {\mathbb R}^{n}$. 
\end{corollary}

\begin{corollary}\label{strong P imply nonempty}
Let $\widehat{\mathcal{A}} = (\mathcal{A}_{0}, \mathcal{A}_{1},...,\mathcal{A}_{k}) \in \Theta^{(k+1)}_{(m,n)}$ and $m$ be even. If $\widehat{\mathcal{A}}$ is a strong $\mathbb{EHP}$ tensor, then ${\rm SOL}(\widehat{\mathcal{A}},\widehat{{\bf d}},{q})$ is  nonempty and compact for each $\widehat{\bf d} \in \Theta^{(k-1)}_{n,++}$ and ${q} \in {\mathbb R}^{n}$. 
\end{corollary}

\subsection{Finiteness and Uniqueness Results}\label{Finiteness and Uniqueness}
This subsection deals with the finiteness of ${\rm SOL}(\widehat{\mathcal{A}},\widehat{{\bf d}},{q})$ with regards to $\mathbb{EHND}$ tensors and strong $\mathbb{EHND}$ tensors. Thereafter, we discuss the uniqueness of solution of the ${\rm EHTCP}(\widehat{\mathcal{A}},\widehat{{\bf d}},{q})$ with respect to a strong $\mathbb{EHP}$ tensor. We first show that the equivalence of the finiteness of  solution set of ${\rm EHTCP}(\widehat{\mathcal{A}},\widehat{\bf d},{q})$ with $\widehat{\mathcal{A}}$ being an $\mathbb{EHND}$ tensor may not hold.  
\begin{example}\rm
 Let $\widehat{\mathcal{A}} = (\mathcal{A}_{0}, \mathcal{A}_{1},\mathcal{A}_{2}) \in \Theta^{(3)}_{(3,3)}$ where $\mathcal{A}_{0} = (a^{0}_{ijk}) \in \mathbb{T}(3,3)$ such that $a^{0}_{111} = a^{0}_{122} = 1$ and other entries are zero, $\mathcal{A}_{1} = (a^{1}_{ijk}) \in \mathbb{T}(3,3)$ with
$a^{1}_{211} = a^{1}_{222} =-1$ and other entries are zero, and  $\mathcal{A}_{2} = (a^{2}_{ijk}) \in \mathbb{T}(3,3)$ having 
$a^{2}_{311} = a^{2}_{322} =1$ and other entries as zero. 
We show that (i) $\widehat{\mathcal{A}}$ is  an  $\mathbb{EHND}$ tensor, (ii) but ${\rm SOL}(\widehat{\mathcal{A}},\widehat{{\bf d}},{q})$ is not finite for some ${q} \in \mathbb{R}^{3}$ and ${d_1} \in \mathbb{R}^{3}_{++}$. For any ${\bf w} = (x,y,z) \in \Theta^{(3)}_{3},$ we have
$$\mathcal{A}_{0}{x^2} = (x_{1}^{2}+x_{2}^{2},0,0)^{T}, \mathcal{A}_{1}{y^{2}} = (0, -(y_{1}^{2}+y_{2}^{2}),0)^{T}~\text{and}~\mathcal{A}_{2}{z^2} = (0,0,z_{1}^{2}+z_{2}^{2})^{T}.$$
\begin{enumerate}
\item[\rm(i)] Let ${\bf w} = (x,y,z) \in \Theta^{(3)}_{3}$ satisfy
$$\mathcal{A}_{0}{x^2} = \mathcal{A}_{1}{y^2}+\mathcal{A}_{2}{z^2}, x \ast y =0, x \ast z=0,~\text{and}~y\ast z =0.$$
A simple calculation yields ${x}= 0, {y} = {0}, {z} = {0}$, and hence ${\bf w} = {\bf 0}.$ Thus  $\widehat{\mathcal{A}}$ is an  $\mathbb{EHND}$ tensor. 
\item[\rm(ii)] Take $q = (1,0,0)^{T} \in \mathbb{R}^{3}$ and $d_{1} = (1,1,1)^{T}$ in $\mathbb{R}^{3}_{++}$. Observe that the vector $(x,y,z) = ((\cos\theta, \sin\theta,0)^{T}, (0, 0,0)^{T}, (0,0,0)^{T}) \in \Theta^{(3)}_{3}$ satisfies
$$\mathcal{A}_{0}{x^2} = q +\mathcal{A}_{1}{y^2} + \mathcal{A}_{2}{z^2}, x \wedge y =0~\text{and}~(d_{1} - y) \wedge z = 0,$$ 
for any $0 \leq \theta \leq \frac{\pi}{2}.$ Therefore ${\rm SOL}(\widehat{\mathcal{A}},\widehat{{\bf d}},{q})$ is not finite.
\end{enumerate} 
\end{example}

\begin{example}\rm\label{not finite}
Let $\widehat{\mathcal{A}} = (\mathcal{A}_{0}, \mathcal{A}_{1}) \in  \Theta^{(2)}_{(4,2)}$ where $\mathcal{A}_{0} = (a^{0}_{ijkl}) \in \mathbb{T}(4,2)$ such that $a^{0}_{1111} =1, a^{0}_{2222} =1$ and other entries are zero, and $\mathcal{A}_{1} = (a^{1}_{ijkl}) \in \mathbb{T}(4,2)$ with $a^{1}_{1111} =1, a^{1}_{1222} =1, a^{1}_{2222} =-1, a^{1}_{2221}=-1$ and other entries are zero. We show that (i) $\widehat{\mathcal{A}}$ is not an $\mathbb{EHND}$ tensor, (ii) but ${\rm SOL}(\widehat{\mathcal{A}},\widehat{{\bf d}},{q})$ is a finite set.   For any ${\bf w} = (x, y) \in \Theta^{(2)}_{2}$, we have
$$\mathcal{A}_{0}{x^3} = (x_{1}^{3}, x_{2}^{3})^{T}~\text{and}~\mathcal{A}_{1}{y^3} = (y_{1}^{3}+y_{2}^{3},-y_{2}^{3}-y_{1}y_{2}^{2})^{T}.$$
\begin{enumerate}
\item[\rm(i)] Observe that the nonzero vector ${\bf w} = ((0,0)^{T},(1,-1)^{T}) \in \Theta^{(2)}_{2}$ satisfies
$$\mathcal{A}_{0}{x^3} = \mathcal{A}_{1}{y^3}~\text{and}~{x} \ast {y} = 0.$$
Therefore $\widehat{\mathcal{A}}$ cannot be an $\mathbb{EHND}$ tensor.
\item[\rm(ii)] We now show that ${\rm SOL}(\widehat{\mathcal{A}},\widehat{\bf d},{q})$ is a finite set. If it is empty, then we are done. Suppose that ${\rm SOL}(\widehat{\mathcal{A}},\widehat{\bf d},{q}) \neq \emptyset$ and let ${\bf w} = (x,y) \in \Theta^{(2)}_{2}$ be a solution of ${\rm EHTCP}(\widehat{\mathcal{A}},\widehat{\bf d},{q}).$ This implies that ${x}\wedge {y} = 0$ and
\begin{eqnarray}\label{ff1}
& \mathcal{A}_{0}{x^3} = {q} + \mathcal{A}_{1}{y^3}\notag\\
& \implies \begin{bmatrix}
x_{1}^{3}\\
x_{2}^{3}
\end{bmatrix} = \begin{bmatrix}
q_{1} + y_{1}^{3}+y_{2}^{3}\\
q_{2}-y_{2}^{3}-y_{1}y_{2}^{2}
\end{bmatrix}.
\end{eqnarray}
Considering ${x} \wedge {y} =0,$ we discuss the following cases:

{Case 1.} If ${x} = (0,0)^{T},$ then ${y} = (y_{1},y_{2})^{T}$ can take the following forms:

{(a)} ${y} = (0,0)^{T}.$

{(b)} ${y} = (y_{1},0)^{T}$ with $y_{1}>0$. Then Eq. (\ref{ff1}) gives $y_{1} =(-q_{1})^{1/3}.$

{(c)} ${y} = (0,y_{2})^{T}$ with $y_{2}>0$. Then Eq. (\ref{ff1}) gives $y_{2} =\{\frac{1}{2}(q_{2}-q_{1})\}^{1/3}.$

{(d)} ${y} = (y_{1},y_{2})^{T}$ with $y_{1} > 0$ and $y_{2}>0$. Then from Eq. (\ref{ff1}), we get 
$$q_{1}+y_{1}^{3}+y_{2}^{3} = 0~\text{and}~q_{2}-y_{2}^{3}-y_{1}y_{2}^{2} = 0.$$
Solving these equations, we obtain a polynomial $[(q_{1}+q_{2})+y_{1}^{3}]^{3}-[y_{1}^{3}(q_{1}+y_{1}^{3})^{2}]$ in terms of $y_{1}.$ It can be easily verified that this polynomial cannot be a zero polynomial for any ${q} \in \mathbb{R}^{2}.$ Since any nonzero polynomial in one variable has finitely many zeros, we get finitely many values of $y_{1}$ and hence $y_{2}$. 

{Case 2.} If ${x} = (x_{1},x_{2})^{T}$ is nonzero, then it can take the following forms:

{(a)} $x_{1} > 0$ and $x_{2} = 0$. From the complementarity condition, we get $y_{1} = 0$ and $y_{2} \geq 0.$ Then Eq. (\ref{ff1}) gives $x_{1}^{3} = q_{1}+y_{2}^{3}$ and $q_{2} - y_{2}^{3}  =0$. Upon solving these equations, we get $y_{2} = (q_{2})^{1/3}$ and $x_{1} = (q_{1}+q_{2})^{1/3}.$ Thus ${x} = ((q_{1}+q_{2})^{1/3},0)^{T}$ and ${y} = (0,(q_{2})^{1/3})^{T}.$

{(b)} $x_{1} = 0$ and $x_{2}>0$. This implies that $y_{1} \geq 0$ and $y_{2} =0.$ Using Eq. (\ref{ff1}) and simplifying, we get ${x} = (0,{q_{2}}^{1/3})^{T}$ and ${y} = ((-q_{1})^{1/3},0)^{T}.$ 

{(c)} $x_{1}>0$ and $x_{2}>0$. This gives ${y} = (0,0)^{T}$. From Eq. (\ref{ff1}), we get ${x} = ({q_{1}}^{1/3},{q_{2}}^{1/3})^{T}.$\\
From all of the above cases, we get that ${\rm SOL}(\widehat{\mathcal{A}},\widehat{\bf d},{q})$ is a finite set. 
\end{enumerate}
\end{example} 

Next, we prove that if $\widehat{\mathcal{A}}$ is a strong $\mathbb{EHND}$ tensor, then ${\rm SOL}(\widehat{\mathcal{A}},\widehat{{\bf d}},{q})$ is  finite  for each $\widehat{\bf d} \in \Theta^{(k-1)}_{n,++}$ and ${q} \in {\mathbb R}^{n}$. We need the following lemma to proceed further.
\begin{lemma}\label{ND11} 
Let $\widehat{\mathcal{A}} = (\mathcal{A}_{0}, \mathcal{A}_{1},...,\mathcal{A}_{k}) \in \Theta^{(k+1)}_{(m,n)}$ be a strong $\mathbb{EHND}$ tensor. Then the following statements are valid.
\begin{enumerate}
\item[\rm(i)] $m$ is even,
\item[\rm(ii)] $F_{i}(x) = \mathcal{A}_{i} {x}^{m-1}$ for each $0 \leq i \leq k $ is an injective function on $\mathbb{R}^{n},$
\item[\rm(iii)]   ${\rm SOL}(\widehat{\mathcal{A}}, \widehat{\bf{d}},{q})$ is compact for all $\widehat{{\bf d}} = (d_{1},d_{2},...,d_{k-1}) \in \Theta^{(k-1)}_{n,++}$ and $q \in \mathbb{R}^{n}$.
\end{enumerate}
\end{lemma}
\begin{proof}
Let $\widehat{\mathcal{A}} = (\mathcal{A}_{0}, \mathcal{A}_{1},...,\mathcal{A}_{k}) \in \Theta^{(k+1)}_{(m,n)}$ be a strong $\mathbb{EHND}$ tensor.\\
(i) Suppose that $m$ is odd. For any nonzero ${\bf x} = (x_{0},0,...,0)$ and $\bar{{\bf x}} = (\bar{x}_{0},0,...,0)$ in $\Theta^{(k+1)}_{n}$ with $\bar{x}_{0} = -x_{0}$, we have
\begin{equation*}
\mathcal{A}_{0}x_{0}^{m-1}-\mathcal{A}_{0}\bar{x}_{0}^{m-1} = 0~\text{and}~(x_{i} - \bar{x}_{i}) \ast (x_{j} - \bar{x}_{j}) = 0~\forall~0\leq i < j \leq k.
\end{equation*}
Due to $\widehat{\mathcal{A}}$ being a strong $\mathbb{EHND}$ tensor, we get ${\bf x} = \bar{\bf x}$. This implies that $x_{0} = {0}$, leading to a contradiction. Hence $m$ must be even.\\
(ii) It is enough to show that $F_{0}(x) = \mathcal{A}_{0} {x}^{m-1}$ is an injective function on $\mathbb{R}^{n}.$ Suppose that $F_{0}(x_{0}) = F_{0}(\bar{x}_{0})$ and $x_{i} = 0 = \bar{x}_{i}$ for all $1 \leq i \leq k$. Then we have
\begin{equation*}
\mathcal{A}_{0}x_{0}^{m-1}-\mathcal{A}_{0}\bar{x}_{0}^{m-1}=0~\text{and}~(x_{i}-\bar{x}_{i})\ast(x_{j}-\bar{x}_{j}) = 0~\forall~0\leq i < j \leq k. 
\end{equation*}
As $\widehat{\mathcal{A}}$ is strong $\mathbb{EHND}$, we get $x_{0} = \bar{x}_{0}$ and hence $F_{0}$ is an injective function. This completes the proof.\\
(iii) As ${\rm SOL}(\widehat{\mathcal{A}},\widehat{{\bf d}},{q})$ is always a closed set, it suffices to show that it is bounded. To the contrary, assume that it is unbounded. Let $\{{\bf w}^{(l)}\}_{l =1}^\infty = \{(w_{0}^{(l)}, {w_{1}^{(l)}},...,{w_{k}^{(l)}})\}_{l=1}^\infty$ be an unbounded sequence in ${\rm SOL}(\widehat{\mathcal{A}},\widehat{{\bf d}},{q})$ for some ${q} \in \mathbb{R}^{n}$ and $\widehat{{\bf d}} = (d_{1},d_{2},...,d_{k-1}) \in \Theta^{(k-1)}_{n,++}.$ So there exists a monotonically increasing subsequence of $\{{\bf w}^{(l)}\}_{l =1}^\infty$ diverging to infinity. Assume (without loss of generality) that $\{{\bf w}^{(l)}\}_{l =1}^\infty$  is itself a monotonically increasing subsequence. Since $\{{\bf w}^{(l)}\}_{l =1}^\infty$  is in ${\rm SOL}(\widehat{\mathcal{A}},\widehat{{\bf d}},{q}),$ we have
\begin{eqnarray*}\label{ff45}
\allowdisplaybreaks
& \mathcal{A}_{0}({w_{0}^{(l)}})^{m-1} = {q} + {\displaystyle\sum_{j = 1}^{k} \mathcal{A}_{j}({w_{j}^{(l)})^{m-1}}}, \notag \\
\allowdisplaybreaks
& {w_{0}^{(l)}}\wedge {w_{1}^{(l)}} = 0~\text{and}~(d_{j} - {w_{j}^{(l)}}) \wedge {w_{j+1}^{(l)}} = 0~\forall~j \in [k-1] . 
\end{eqnarray*}
As $\{{\bf w}^{(l)}\}_{l =1}^\infty \in {\rm SOL}(\widehat{\mathcal{A}},\widehat{{\bf d}},{q})$,  Lemma \ref{solution} gives
\begin{equation}\label{fdfd1}
\mathcal{A}_{0}({w_{0}^{(l)}})^{m-1} = {q} + {\sum_{j = 1}^{k} \mathcal{A}_{j}({w_{j}^{(l)})^{m-1}}}~\text{and}~{w_{0}^{(l)}}\wedge {w_{j}^{(l)}} = 0~\forall~j \in [k]. 
\end{equation}
Since $\|{\bf w}^{(l)}\| \to \infty$ as $l \to \infty$, $\dfrac{{\bf w}^{(l)}}{\|{\bf w}^{(l)}\| }$ is a unit vector for all large $l$. This implies that $\dfrac{{\bf w}^{(l)}}{\|{\bf w}^{(l)}\| }$  converges to some ${\bf z} = (z_{0},z_{1},...,z_{k}) \in \Theta^{(k+1)}_{n}$ with $\|{\bf z}\| =1.$ Moreover, observe that $(d_{j} - {w_{j}^{(l)}}) \geq {0}~\forall~j \in [k-1].$ So $0 \leq {w_{j}^{(l)}} \leq d_{j}~\forall~j\in[k-1].$ Hence, we have
\begin{equation}\label{fff46}
{z_{j}} = \lim_{l \to \infty} \dfrac{{w_{j}^{(l)}}}{\|{\bf w}^{(l)}\|} =0~\forall~j\in[k-1]. 
\end{equation}
Dividing Eq. (\ref{fdfd1}) by ${\|{\bf w}^{(l)}\|}$ and taking limit $l \to \infty$, we get
\begin{equation*}
\mathcal{A}_{0}{z_{0}^{m-1}} = {\sum_{j = 1}^{k} \mathcal{A}_{j}{z_{j}^{m-1}}}~\text{and}~{z_{0}} \wedge z_{j} = 0~\forall~j \in [k].
\end{equation*}
This implies that $z_{0} \ast z_{k} = 0.$ Due to Eq. (\ref{fff46}), ${\bf z}$ takes the form ${\bf z} = (z_{0},0,...,0,z_{k})$ with $z_{i} \ast z_{j} =0~\forall~0 \leq i < j \leq k.$ Therefore, we have
\begin{equation}\label{qq1}
\mathcal{A}_{0}{z_{0}^{m-1}} = {\sum_{j = 1}^{k} \mathcal{A}_{j}{z_{j}^{m-1}}}~\text{and}~z_{i} \ast z_{j} =0~\forall~0 \leq i < j \leq k.
\end{equation}
Note that $\|{\bf z}\| =1.$ So a nonzero vector ${\bf z}$ satisfies Eq. (\ref{qq1}), which implies that $\widehat{\mathcal{A}}$ cannot be an $\mathbb{EHND}$ tensor. This contradicts our assumption of $\widehat{\mathcal{A}}$ being a strong $\mathbb{EHND}$ tensor. Hence  ${\rm SOL}(\widehat{\mathcal{A}}, \widehat{\bf{d}},{q})$ is compact for all $\widehat{{\bf d}} = (d_{1},d_{2},...,d_{k-1}) \in \Theta^{(k-1)}_{n,++}$ and $q \in \mathbb{R}^{n}$.   
\end{proof}
We now state our result related to the finiteness of the ${\rm SOL}(\widehat{\mathcal{A}},\widehat{{\bf d}},{q})$. We skip the proof as it follows along the similar lines as in \cite[Theorem 2]{YADAV2024107112}.

\begin{theorem}\label{strong ND imply finiteness}
Let $\widehat{\mathcal{A}} = (\mathcal{A}_{0}, \mathcal{A}_{1},...,\mathcal{A}_{k}) \in \Theta^{(k+1)}_{(m,n)}$. If $\widehat{\mathcal{A}}$ is a strong $\mathbb{EHND}$ tensor, then ${\rm SOL}(\widehat{\mathcal{A}},\widehat{{\bf d}},{q})$ is  finite  for each $\widehat{\bf d} \in \Theta^{(k-1)}_{n,++}$ and ${q} \in {\mathbb R}^{n}$.
\end{theorem}

\begin{example}\rm
The converse of the Theorem \ref{strong ND imply finiteness} does not hold (see Example \ref{not finite}).
\end{example}

In the following, we address the uniqueness of the solution of the ${\rm EHTCP}(\widehat{\mathcal{A}},\widehat{{\bf d}},{q})$ with respect to $\widehat{\mathcal A}$ being a strong $\mathbb{EHP}$ tensor.
\begin{lemma}\label{EHP11}
Let $\widehat{\mathcal{A}} = (\mathcal{A}_{0}, \mathcal{A}_{1},...,\mathcal{A}_{k}) \in \Theta^{(k+1)}_{(m,n)}$ be a strong $\mathbb{EHP}$ tensor. Then the following statements are valid.
\begin{enumerate}
\item[\rm(i)] $m$ is even,
\item[\rm(ii)] $F_{i}(x) = \mathcal{A}_{i} {x}^{m-1}$ for each $0 \leq i \leq k $ is an injective function on $\mathbb{R}^{n},$
\item[\rm(iii)] $\deg(\widehat{\mathcal{A}}) \neq 0.$
\end{enumerate}
\end{lemma}
\begin{proof}
Let $\widehat{\mathcal{A}} = (\mathcal{A}_{0}, \mathcal{A}_{1},...,\mathcal{A}_{k}) \in \Theta^{(k+1)}_{(m,n)}$ be a strong $\mathbb{EHP}$ tensor. From Proposition \ref{subset} (iii), it follows that $\widehat{\mathcal{A}}$ is a strong $\mathbb{EHND}$ tensor. Therefore (i) and (ii) immediately follow from (i) and (ii) of Lemma \ref{ND11}, respectively.\\
(iii) Since $\widehat{\mathcal{A}}$ is a  strong $\mathbb{EHP}$ tensor,  $\widehat{\mathcal{A}}$ is an $\mathbb{EHE}$ tensor and $m$ is even, we get $\deg(\widehat{\mathcal{A}}) \neq 0$ from Theorem \ref{E imply nonempty}.  
\end{proof}

\begin{theorem}
Suppose that $\widehat{\mathcal{A}} = (\mathcal{A}_{0}, \mathcal{A}_{1},...,\mathcal{A}_{k}) \in \Theta^{(k+1)}_{(m,n)}$  is a strong $\mathbb{EHP}$ tensor. Then ${\rm EHTCP}(\widehat{\mathcal{A}},\widehat{{\bf d}},{q})$ has a unique solution   for each $\widehat{\bf d} \in \Theta^{(k-1)}_{n,++}$ and ${q} \in {\mathbb R}^{n}$. 
\end{theorem}
\begin{proof}
Let $\widehat{\mathcal{A}} = (\mathcal{A}_{0}, \mathcal{A}_{1},...,\mathcal{A}_{k})$ be in $\Theta^{(k+1)}_{(m,n)}$ such that  $\widehat{\mathcal{A}}$ is a strong $\mathbb{EHP}$ tensor. The solvability of  ${\rm EHTCP}(\widehat{\mathcal{A}},\widehat{{\bf d}},{q})$  follows from Corollary \ref{strong P imply nonempty}.  
Let ${\bf x} = (x_{0},x_{1},...,x_{k}), \bar{\bf x} = (\bar{x}_{0},\bar{x}_{1},...,\bar{x}_{k})$  in $\Theta^{(k+1)}_{n}$ be  two solutions of the ${\rm EHTCP}(\widehat{\mathcal{A}},\widehat{{\bf d}},{q})$. Then from Lemma \ref{solution}, we have
\begin{equation*}
\begin{aligned}
\begin{cases}
\mathcal{A}_{0}x_{0}^{m-1} = {q} + \displaystyle {\sum_{j =1}^k} \mathcal{A}_{j}x_{j}^{m-1}\\x_{0} \wedge x_{j} = 0~\forall~j \in [k]
\end{cases} 
\text{and  } & 
\begin{cases}
\mathcal{A}_{0}\bar{x}_{0}^{m-1} = {q} + \displaystyle {\sum_{j =1}^k} \mathcal{A}_{j}\bar{x}_{j}^{m-1}\\\bar{x}_{0} \wedge \bar{x}_{j} = 0~\forall~j \in [k].
\end{cases}
\end{aligned}
\end{equation*}
Equivalently, we get
\begin{equation*}
\left.
\begin{aligned}
(\mathcal{A}_{0}x_{0}^{m-1}- \mathcal{A}_{0}\overline{x}_0^{m-1}) = {\sum_{j =1}^k (\mathcal{A}_{j}x_j^{m-1}}-\mathcal{A}_{j}\overline{x}_j^{m-1}),\\
(x_{0}- \overline{x}_{0}) \ast (x_{i}-\overline{x}_{i}) \leq {0}~\forall~i \in [k].
\end{aligned}
\right.
\end{equation*}
As $\widehat{\mathcal{A}}$ is strong $\mathbb{EHP}$, we see that ${\bf x} = \bar{\bf x}$. This implies that ${\rm EHTCP}(\widehat{\mathcal{A}},\widehat{{\bf d}},{q})$ has a unique solution   for each $\widehat{\bf d} \in \Theta^{(k-1)}_{n,++}$ and ${q} \in {\mathbb R}^{n}$.   
\end{proof}

 \section{Conclusions}\label{section 5}
 In this paper, we introduced the extended horizontal tensor complementarity problem (EHTCP) and investigated the properties of the solution set of the EHTCP. By defining new structured tensors, namely $\mathbb{EHR}_{0}$ tensor, $\mathbb{EHP}$ tensor, and $\mathbb{EHE}$ tensor, we obtained the nonemptiness and compactness of the solution set of the EHTCP using degree theory. Moreover, we proved that under the condition of a strong $\mathbb{EHP}$ tensor, the solution to the EHTCP is unique. Finally, we explored the finiteness of the solution set of EHTCP by introducing the concepts of $\mathbb{EHND}$ and strong $\mathbb{EHND}$ tensors, and established that for a strong $\mathbb{EHND}$ tensor, the EHTCP has a finite solution set.

\end{document}